\documentclass{article}

\usepackage{arxiv}

\usepackage[utf8]{inputenc} 
\usepackage[T1]{fontenc}    
\usepackage{hyperref}       
\usepackage{url}            
\usepackage{booktabs}       
\usepackage{amsfonts}       
\usepackage{mathrsfs}         
\usepackage{amsmath,amssymb,bbm}
\usepackage{amsthm}
\usepackage{microtype}      
\usepackage{tikz-cd}
\usepackage{quiver}

\newtheorem{theorem}{Theorem}[section]

\newtheorem{proposition}[theorem]{Proposition}
\newtheorem{lemma}[theorem]{Lemma}
\newtheorem{corollary}[theorem]{Corollary}
\newtheorem{conjecture}[theorem]{Conjecture}

\newtheorem{example}{Example}
\newtheorem{remark}{Remark}

\newcommand{\MW}{Milnor-Witt\ }
\newcommand{\rMW}{\mathrm{MW}}
\newcommand{\KMW}{\mathrm{K}^\mathrm{MW}}

\newcommand{\tbb}[1]{\tilde{\mathbb{#1}}}
\newcommand{\wt}[1]{\widetilde{#1}}
\newcommand{\Spec}{\mrm{Spec}\ }
\newcommand{\af}{\mathbb{A}}
\newcommand{\afnz}[1]{\mathbb{A}^{#1}\setminus \{0\}}
\newcommand{\nonZero}{\setminus \{0\}}

\newcommand{\bZ}{\mathbb{Z}}
\newcommand{\tbZ}{\tbb{Z}}
\newcommand{\mrm}[1]{\mathrm{#1}}
\newcommand{\inv}{^{-1}}
\newcommand{\Gm}{\mathbb{G}_m}

\newcommand{\Sp}{\mrm{Sp}}
\newcommand{\GL}{\mrm{GL}}
\newcommand{\SL}{\mrm{SL}}
\newcommand{\MSp}{\mrm{MSp}}
\newcommand{\BSp}{\mrm{BSp}}
\newcommand{\SH}{\mathcal{SH}}
\newcommand{\Da}{D_{\af^1}}
\newcommand{\DM}{\mrm{DM}}
\newcommand{\DMt}{\widetilde{\mathrm{DM}}}

\newcommand{\Mt}{\wt{\mrm{M}}}
\newcommand{\Meta}{\wt{\mrm{M}}_{\eta}}
\newcommand{\HH}{\mathrm{H}_{\rMW}}
\newcommand{\Heta}{\mathrm{H}_{\eta}}

\newcommand{\aBra}[1]{\left<#1\right>}

\newcommand{\rHom}{\mathrm{Hom}}

\newcommand{\xr}[1]{\xrightarrow{#1}}

\title{\MW motivic cohomology and linear algebraic groups}

\date{\ } 					

\author{Keyao Peng}

\begin{document}
\maketitle
\begin{abstract}
	This article presents two key computations in MW-motivic cohomology. Firstly, we compute the MW-motivic cohomology of the symplectic groups $\Sp_{2n}$ for any $n\in\mathbb{N}$ using the $\Sp$-orientation and the associated Borel classes. 
	
	Secondly, following the classical computations and using the analogue in $\af^1$-homotopy of the Leray spectral sequence, we compute the $\eta$-inverted MW-motivic cohomology of general Stiefel varieties, obtaining in particular the computation of the $\eta$-inverted MW-motivic cohomology of the general linear groups $\GL_n$ and the special linear groups $\SL_n$ for any $n\in\mathbb{N}$.
	
	Finally, we determine the multiplicative structures of these total cohomology groups.
\end{abstract}
\tableofcontents
\section{Introduction}

This article aims to provide a comprehensive overview of the interplay between \MW motivic cohomology and linear algebraic groups and to demonstrate the usefulness of \MW motivic cohomology in understanding the properties and structure of these groups. Linear algebraic groups play a fundamental role in various mathematical disciplines. In order to study them, cohomology has emerged as a critical invariant. In this article, we compute the cohomology of an algebraic group $G$, seen as a variety, using the \MW motivic cohomology $\HH$ instead of employing the more commonly used group cohomology of the classifying space $BG$ (the $BG$ here refers to a geometric model described in \cite[\S 4.2]{morel19991}), although the cohomology of $G$ and $BG$ are connected (section \ref{ComputeSp}).

For context, recall that the \MW motivic cohomology $\HH$ is defined as the unit object in the tensor triangulated category $\DMt(S)$ of Milnor-Witt motives over $S$ \cite{bachmann2020milnor}, and can be seen in a precise sense as an analogue of singular cohomology in topology. In the sequel, we denote by $\Mt(X)$ the \MW motive of a scheme $X$. If $x\in X(S)$, we denote by $\Mt(X,x)$ the reduced \MW motive of $X$. For instance, if $n\geq 1$, we obtain for $\afnz{n}:=\mathbb{A}^n_S\setminus \{0\}$ the following motives
\[
\Mt(\afnz{n})=\Mt(S)\oplus \Mt(S)(n)[2n-1],
\]
while $\Mt(\afnz{n},x)=\Mt(S)(n)[2n-1]$ for any $S$-rational point $x$.

For the original motivic cohomology, the computations and motivic decompositions of split reductive groups \cite{biglari2012motives} and Stiefel varieties \cite{williams2012motivic} in motivic cohomology are well known. At present, there are only limited results for \MW motivic cohomology and the main goal of this article is to fill this gap. We present two original results using different methods. Firstly, we compute the \MW motivic decomposition of $\Sp_{2n}$ in $\DMt(S)$, over any reasonable base $S$, by utilizing the $\Sp$-orientation of \MW motivic cohomology. The final result is the following. 
\begin{theorem}[Theorem \ref{SpMain}]
	 For any $n\geq 1$, We have an isomorphism in $\DMt(S)$
	\[ \Mt(\Sp_{2n})\cong  \Mt(\afnz{2n})\otimes \Mt(\Sp_{2n-2}),\]
	yielding a decomposition in $\DMt(S)$ of the form
	\[\Mt(\Sp_{2n})\cong \Mt(\afnz{2n})\otimes \Mt(\afnz{2n-2})\otimes\ldots \otimes \Mt(\afnz{2}) \cong \bigoplus_{1\leq i_1<\ldots<i_j \leq n} \tbZ(d(i_1,\ldots,i_j))  .\] 
	where $\tbZ(d(i_1,\ldots,i_j))$ denotes $\tbZ(\sum_{l=1}^j 2i_l)[4(\sum_{l=1}^j i_l) -j] $.
\end{theorem}

Next, we address the computation for $\GL_n$ and $\SL_n$. The fundamental difference between \MW motivic cohomology and its more classical version is that the former is not $\GL$-oriented (i.e. doesn't have a reasonable theory of Chern classes), thus making the computations more delicate. For this article, the $\eta$-inverted Milnor-Witt cohomology $\Heta$ provides the most effective mean of computations. This allows us to parallel the classical computations for both $O_n$ and $SO_n$ \cite{hatcher2002algebraic}. Using the Leray spectral sequence, we can inductively compute for all Stiefel varieties. In the next statement, we write $V_{k}(\af^{n})$ for the Stiefel variety of full rank $k\times n$ matrices, where $k\leq n$ and we set
\begin{eqnarray*}
\mathbf{HS}_{2k}&:=&\Meta(\afnz{2k})\cong \Meta(S)\oplus \Meta(S)(2k)[4k-1], \\
 \mathbf{HS}_{2k+1}&:=& \Meta(S)\oplus \Meta(S)(4k+1)[8k].
\end{eqnarray*}
Here, let us note that $\Meta(S)(1)[1]$ is canonically isomorphic to $\Meta(S)$, via multiplication by $\eta$, and we could therefore have set
\begin{eqnarray*}
\mathbf{HS}_{2k}\cong\Meta (S)\oplus \Meta(S)[2k-1], \\
 \mathbf{HS}_{2k+1}\cong \Meta(S)\oplus \Meta(S)[4k-1].
\end{eqnarray*}
We however prefer to keep the previous notation, mainly since we plan to compute the integral motive of Stiefel varieties in $\DMt(S)$ in the near future.

\begin{theorem}[Theorem \ref{SVCohMain}]
	We have an isomorphism of graded algebras
	\[\Heta^{*}(V_{k}(\af^{n}))\cong \Heta^{*}(S)[G]/(g^2-\delta_{\gamma_0}(g)[-1]g , g\in G)\] 
	where the generators $G$ are given by 
	\begin{equation*}
		G=\begin{cases}
			G_1\sqcup G_2\sqcup G_3 &\text{$n$ even, $n-k$ even}\\
			G_2\sqcup G_3 &\text{$n$ odd, $n-k$ even}\\
			G_1\sqcup G_2 &\text{$n$ even, $n-k$ odd}\\
			G_2 &\text{$n$ odd, $n-k$ odd}\\
			\end{cases},
	\end{equation*}
	here 
	\begin{eqnarray*}
		G_1&=&\{\alpha_{n-1}\in \Heta^{n-1}(\mathbf{HS}_n)\}\\
		G_2&=&\{\beta_{4j-1}\in \Heta^{4j-1}(\mathbf{HS}_{2j+1})|n-k<2j<n\}\\
		G_3&=&\{\gamma_{n-k} \in \Heta^{n-k}(\afnz{n-k+1})\}
	\end{eqnarray*}
	And $\delta_{\gamma_0}(\gamma_0)=1$, otherwise $ \delta_{\gamma_0}(g)=0$.
\end{theorem}

As a consequence, we obtain the following theorem.
\begin{theorem}[Theorem \ref{SVMain}]
We have a motivic decomposition of the following form for any $i,j\in\mathbb{N}$:
	\[\Meta(V_{2j}(\af^{2i})) \cong \mathbf{HS}_{2i} \otimes  \mathbf{HS}_{2i-1}\otimes \mathbf{HS}_{2i-3}\otimes \ldots \otimes \mathbf{HS}_{2i-1-2(j-2)} \otimes \Meta(\afnz{2i+1-2j}) \]
	\[\Meta(V_{2j+1}(\af^{2i})) \cong \mathbf{HS}_{2i} \otimes \mathbf{HS}_{2i-1}\otimes \mathbf{HS}_{2i-3}\otimes \ldots \otimes \mathbf{HS}_{2i-1-2(j-1)}  \]
	\[\Meta(V_{2j}(\af^{2i+1})) \cong \mathbf{HS}_{2i+1}\otimes \mathbf{HS}_{2i-1}\otimes \ldots \otimes \mathbf{HS}_{2i+1-2(j-1)}  \]
	\[\Meta(V_{2j+1}(\af^{2i+1})) \cong \mathbf{HS}_{2i+1}\otimes \mathbf{HS}_{2i-1}\otimes \ldots \otimes \mathbf{HS}_{2i+1-2(j-1)} \otimes \Meta(\afnz{2i+1-2j}).  \]
\end{theorem}

In the first part of our study, we employ the $\Sp$-orientation and Borel classes as described in \cite{panin2022quaternionic}. In the second part, we make use of the homotopy Leray spectral sequence \cite{asok2018homotopy}, which is a significant tool in our computations. It is worth noting that our results represent the first non-trivial application of this spectral sequence in computations. Next, we establish the multiplicative structure with the aid of elementary arguments on the degree of the generators.

When considering rational coefficients, we can draw on the result of reconstructing rational motives presented in \cite{garkusha2019reconstructing}: 

$$\SH(K)_{\mathbb{Q}}\cong\DMt(K)_{\mathbb{Q}}\cong \DM(K)_{\mathbb{Q}} \times \DMt(K)[\eta^{-1}]_{\mathbb{Q}}$$ 

Consequently, we are able to fully comprehend a rational motive in $ \DMt(S)_{\mathbb{Q}}$ by either setting $\eta=0$ or making $\eta$ invertible. In our case, when combined with the results on motivic cohomology of Stiefel varieties in \cite{williams2012motivic} (i.e., the $\eta=0$ case), we can understand the rational motive $\Mt(V_k(\af^n))_{\mathbb{Q}}\in \DMt(S)_{\mathbb{Q}})$.

However, for integral cases, it remains generally unknown how to recover $\Mt(X)$ from $\mathrm{M}(X)$ and $\Meta(X)$. Some progress has been made when $X$ is Tate \cite{fasel2023tate}.

As a continuation of this work, it would be possible to study the integral MW-motivic cohomology of Stiefel varieties. We conjecture that $\Mt(V_k(\af^n))$ is the direct sum of $\Mt(S)$ and the cone of $\eta$ with certain shifts. This hypothesis leads us to the following conjecture:

\begin{conjecture}
	\label{IntegralConj}
Let $C_\eta(S)$ be the cone of $\Mt(S) \xr{\eta}  \Mt(S)(1)[1]$ and
\begin{eqnarray*}
\wt{\mathbf{HS}}_{2k}&:=& \Mt(\afnz{2k})\cong \Mt (S)\oplus \Mt(S)(2k)[4k-1], \\
 \wt{\mathbf{HS}}_{2k+1}&:=& \Mt(S)\oplus C_\eta(S)(2k)[4k] \oplus \Mt(S)(4k+1)[8k] .
\end{eqnarray*}
Then we have a motivic decomposition of the following form for any $i,j\in\mathbb{N}$:
	\[\Mt(V_{2j}(\af^{2i})) \cong \wt{\mathbf{HS}}_{2i} \otimes  \wt{\mathbf{HS}}_{2i-1}\otimes \wt{\mathbf{HS}}_{2i-3}\otimes \ldots \otimes \wt{\mathbf{HS}}_{2i-1-2(j-2)} \otimes \Mt(\afnz{2i+1-2j}) \]
	\[\Mt(V_{2j+1}(\af^{2i})) \cong \wt{\mathbf{HS}}_{2i} \otimes \wt{\mathbf{HS}}_{2i-1}\otimes \wt{\mathbf{HS}}_{2i-3}\otimes \ldots \otimes \wt{\mathbf{HS}}_{2i-1-2(j-1)}  \]
	\[\Mt(V_{2j}(\af^{2i+1})) \cong \wt{\mathbf{HS}}_{2i+1}\otimes \wt{\mathbf{HS}}_{2i-1}\otimes \ldots \otimes \wt{\mathbf{HS}}_{2i+1-2(j-1)}  \]
	\[\Mt(V_{2j+1}(\af^{2i+1})) \cong \wt{\mathbf{HS}}_{2i+1}\otimes \wt{\mathbf{HS}}_{2i-1}\otimes \ldots \otimes \wt{\mathbf{HS}}_{2i+1-2(j-1)} \otimes \Mt(\afnz{2i+1-2j}).  \]
\end{conjecture}

However, the multiplicative structures of cohomology groups are quite intricate and depend on the results of Bockstein homomorphisms.

\subsection*{Acknowledgement}

This work constitutes the major part of my PhD thesis. I would like to express my deepest gratitude to my advisor, Jean Fasel, for his invaluable guidance and support throughout my PhD journey. I am also grateful to Institut Fourier for hosting my PhD research. Finally, I extend my thanks to Frédéric Déglise for the helpful discussions and assistance.

\subsection*{Notations and conventions}

The letter $K$ will always denote a perfect field of characteristic different from $2$, while $S$ will denote a fixed Noetherian base scheme over $K$. Further, we will denote by $\mathrm{Sm}/S$ the category of smooth schemes of finite type over $ S $. The isomorphisms $\cong$ and commutative diagrams utilized in this article are typically understood to be up to homotopy or weak equivalence, unless otherwise specified.

\paragraph*{$\af^1$-homotopy and stable $\af^1$-homotopy categories.}

We define the $\af^1$-homotopy category, denoted by $\mathcal{H}(S)$, as 
\[ \mathcal{H}(S):=\mathrm{L}_{Nis,\af^1}PShv(\mathrm{Sm}/S)\]
Here, $ PShv(\mathrm{Sm}/S) $ denotes the category of presheaves of simplicial sets on $\mathrm{Sm}/S$, and $\mathrm{L}_{Nis,\af^1} $ is the localization functor with respect to the Nisnevich topology and $\mathbb{A}^1$-invariance. 

The $\af^1$-stable homotopy category, denoted by $\SH(S)$, is obtained by stabilizing the $\af^1$-homotopy category with respect to $\Omega_{\mathbb{P}^1} $:
\[\Omega_{\mathbb{P}^1}^{\infty}: \SH(S)\rightleftarrows \mathcal{H}_*(S): \Sigma_{\mathbb{P}^1}^{\infty} \]

Over the base field $K$, there is $\af^1$-derived category, denoted by $\Da(K)$ \cite[\S 5]{morel2012a1}, and we have another adjunction:

\[C^{\af^1}_*: \mathcal{H}(K)\rightleftarrows \Da(K): K^{\af^1} \]

\paragraph*{Cohomology theories.}

Given a spectrum $E \in \SH(S)$, we can define a bigraded cohomology theory for smooth $S$-schemes, as follows:
\[E^{p,q}(X):=[\Sigma_{\mathbb{P}^1}^{\infty}X_+, \Sigma^{p,q}E]_{\SH(S)} \]
where $\Sigma^{p,q}E:= S^{p,q}\wedge E= S^{p-q} \wedge \Gm^{q}\wedge E$.

In our analysis, we focus on ring spectra, i.e. spectra endowed with a multiplication $m: E\wedge E \to E$ satisfying the usual axioms. This multiplication induces a cup product $ \cup: E^{*,*}(X)\otimes E^{*,*}(X) \to E^{*,*}(X)$. It is worth noting that $E^{*,*}(X)$ is a bigraded $(-1,\aBra{-1})$-commutative ring, where $\aBra{-1}=1+\eta [-1] \in \KMW_0$. That is, for any $a\in E^{p,q}(X)$ and $b\in E^{s,t}(X)$, we have $a\cup b= (-1)^{p+s}\aBra{-1}^{q+t}b\cup a$. All the cohomology rings we compute here are bigraded $(-1,\aBra{-1})$-commutative rings.

\paragraph*{\MW motives.}
The category of \MW motives $\DMt(S)$ is equipped with a pair of adjoint functors
$$\wt{\gamma}_*: \SH(S) \rightleftarrows \DMt(S): \wt{\gamma}^*.$$
And over $K$ we have a similar pair of adjoint functors with $\Da(K)$ \cite[\S 1.3 (\textbf{MW4})]{bachmann2020milnor}.

As written above, the \MW motivic spectrum $\HH$ can be defined as $\HH:= \wt{\gamma}^*(\tbZ) $, where $ \tbZ\in \DMt(S)$ is the unit object. Let $  \tbZ(q)[p]:= \wt{\gamma}_*(S^{p,q})$ denote the Tate twists and shifts and let $ \Mt(X):=\wt{\gamma}_*(\Sigma_{\mathbb{P}^1}^{\infty}X_+)$ denote the \MW motive. The \MW motivic cohomology is given by 
\[
	\HH^{p,q}(X):=[\Sigma_{\mathbb{P}^1}^{\infty}X_+, \Sigma^{p,q}\HH]_{\SH(S)}\cong\rHom_{\DMt(S)}(\Mt(X), \tbZ(q)[p])
\]
One important fact is that, for $q<0$ 
\[\HH^{p,q}(X)\cong \mathrm{H}^{p-q}(X,\mathbf{W})\]
where $\mathbf{W}$ represents the Witt sheaf.

In a tensor category $\mathcal{C}$, we have the flexibility of omitting the tensor unit when expressing tensor products, denoted as $ A\otimes B:= A \otimes_{\mathbbm{1}} B$. In the context of $\DMt(S)$, the tensor unit takes the form of $\Mt(S)$. Hence, we can simply write $A\otimes B $ for $A \otimes_{\Mt(S)} B$.

\paragraph*{$\eta$-inverted cohomology.}
Consider the map $\eta: S^{3,2} \cong \Sigma_{\mathbb{P}^1}^{\infty}\afnz{2} \to \Sigma_{\mathbb{P}^1}^{\infty}\mathbb{P}^1 \cong S^{2,1}$, known as the $\af^1$ Hopf map. It is clear that $\eta$ induces a morphism $\eta: E^{p,q}(X) \to E^{p-1,q-1}(X)$. By localizing with respect to $\eta$, we obtain the $\eta$-inverted cohomology theory $E[\eta^{-1}]$, which has the property that the multiplication by $\eta: E[\eta^{-1}]^{p,q}(X) \xrightarrow{\sim} E[\eta^{-1}]^{p-1,q-1}(X)$ an isomorphism.

In the case of $\HH$, since $X \mapsto \HH^{p,q}(X)[\eta^{-1}] \cong \mathrm{H}^{p-q}(X,\mathbf{W})$ for $q<0$, we see that $\Heta:=\HH[\eta^{-1}]$ represents the cohomology in the Witt sheaf, in the sense that
\[
[\Sigma_{\mathbb{P}^1}^{\infty}X_+, \Sigma^{p,q}\Heta]_{\SH(S)}=\mathrm{H}^{p-q}(X,\mathbf{W}).
\]

More generally, for an $\eta$-inverted spectrum $E_{\eta}$, we can obtain single graded cohomology groups $E_{\eta}^*(X) := \oplus_{n\in \mathbb{Z}}E_{\eta}^{n+*,n}(X) \cong E_{\eta}^{*,0}(X)[\eta,\eta^{-1}]$.

\section{Computations on Symplectic Groups}
\label{ComputeSp}

In this section, our focus is on the additive structure of the cohomology groups of general $\Sp$-oriented theories. It is important to note that all the isomorphisms discussed in this section are isomorphisms of modules. In section \ref{ring}, we will study the multiplicative structure, specifically for $\HH$.

\subsection{Recollections on oriented cohomology theories.}

In our context, a $G$-orientation of a cohomology theory (i.e. a spectrum) $E$ is simply a morphism $\tau: MG \to E$, where $MG$ is the Thom spectrum associated with the classifying space $BG$ of the group $G$. If such an orientation exists, we say that $E$ is $G$-oriented. One of the main cases of interest is the $\Sp$-oriented theory \cite{panin2023algebraic}.

The canonical morphism 
$$B\Sp_{2n} \to M\Sp_{2n} \to \Sigma^{4n,2n} M\Sp$$ 
gives rise to a universal Euler class $\theta \in M\Sp^{4n,2n}(B\Sp_{2n})$, which can also be defined for $\Sp$-oriented theories as $\theta \in E^{4n,2n}(B\Sp_{2n})$ through the orientation.

In fact, for $\Sp$-oriented theories, we have the splitting principle, which allows us to define Borel classes (e.g. \cite{panin2022quaternionic})
\[
b_j \in E^{4j,2j}(\mathrm{BSp}_{2n})
\] 
for $1\leq j \leq n$, satisfying $E^{*,*}(B\Sp_{2n}) \cong E^{*,*}(K)[[b_1, \ldots, b_n]]$ with $b_n = \theta$.

In particular, the \MW motivic cohomology $\HH$ is $\SL$-oriented and $\Sp$-oriented but not $\GL$-oriented.

\subsection{Leray-Hirsch theorem}
Let $E\in\mathcal{SH}(S)$ be a ring spectrum, and let $G$ be a smooth scheme over $S$. We aim to compute the cohomology ring $E^{*,*}(G)=\oplus E^{p,q}(G)$ as an $E^{*,*}(S)$-algebra. 
However, determining the multiplicative structure for a general $E$ is often difficult. Our focus here is on additive structure.

\begin{example}\label{caseAn}
	Consider $G=\mathbb{A}^n_S \nonZero$. Since $E^{p,q}((\mathbb{A}^n_S \nonZero))\cong E^{p,q}(S)\oplus E^{p-(2n-1),q-n}(S)$, we obtain  
	$$E^{*,*}(\mathbb{A}^n_S \nonZero)\cong E^{*,*}(S)\cdot 1\oplus E^{*,*}(S) \cdot \theta_n \cong E^{*,*}(S) [\theta_n]/(\theta_n^2) $$ as $ E^{*,*}(S)$-module, 
	 where  $\theta_n\in E^{2n-1,n}((\mathbb{A}^n_S \nonZero))$ is the canonical element induced by the unit of the spectrum.
\end{example}

To compute the cohomology ring $E^{*,*}(G_n)$ for algebraic groups, we define a family $\{G_n\}, n\in \mathbb{N}$ of $S$-schemes to be a \textbf{ $k$-fiber family} over $S$ for some $k\in\mathbb{N^*}$ if there exists for each $G_n$ a Zariski locally trivial fibration
\[ 
G_{n-1} \hookrightarrow G_n \xr{f_n} \af^{kn}_S \nonZero,
\] 
with $G_0 = S$ and a convenient choice of base point for $\af^{kn}_S\nonZero$. For any morphism $T\to S$, the base-changed family $G_{n,T} = G_n \times_{S} T $ is also a $k$-fiber family over $T$.
\begin{example}
  $\{\GL_{n}\times S\vert n\in\mathbb{N}\}$ is a $1$-fiber family over $S$, and $\{\Sp_{2n}\times S\vert n\in\mathbb{N}\}$ is a $2$-fiber family over $S$ \cite[\S 2]{suslin1991k}.
\end{example}

We can now state our theorem (proved below) for computing the cohomology ring $E^{*,*}(G_n)$ of a $k$-fiber family, which is a version of Leray-Hirsch theorem based on Suslin's result \cite{suslin1991k}.

\begin{theorem}\label{susMain}
	Consider a $k$-fiber family over $S$ and a ring spectrum $E\in\mathcal{SH}(S)$. Assume that for all $n$, $ G_n $ has series of classes $ c(n)_j \in E^{2kj-1,kj}(G_n)$, where $ 1\leq j\leq n $, such that:
	\begin{enumerate}
		\item $c(n)_n = u f_n^*(\theta_{kn})$, where $u\in E^{0,0}(S)$ is a unit;
		\item for any fiber embedding $ i: G_{n-1} \hookrightarrow G_n $, the generators $ (i^*(c(n)_1),\ldots,i^*(c(n)_{n-1})) $ are isomorphic to $(c(n-1)_1,\ldots,c(n-1)_{n-1})$ up to an isomorphism in $ GL_{n-1}(\mathbb{Z}) $.
	\end{enumerate}
	Then the total cohomology group $E^{*,*}(G_n)$ isomorphic to a freely generated square-zero algebra as a module.
	\[ E^{*,*}(G_n)\cong E^{*,*}(S) [c(n)_i]/(c(n)_i^2, 1\leq i\leq n) \] 
	and hence, it is free as $ E^{*,*}(S) $-module with basis $ c(n)_{i_1}\cup\ldots\cup c(n)_{i_j}, 1\leq i_1<\ldots<i_j \leq n $ and $0\leq j \leq n$.
\end{theorem}

Since the cohomology ring is free as a module, we have the Künneth formula and the following corollary holds.
\begin{corollary}
	The base changed $k$-fiber family $G_{n,T} = G_n \times_{S} T $ and the classes $ p_2^*(c(n)_{i})$ also satisfy the conditions of theorem \ref{susMain}, and we have 
	\[E^{*,*}(G_{n,T})\cong E^{*,*}(T)\otimes_{E^{*,*}(S)} E^{*,*}(G_{n})\]
\end{corollary}
Giving a morphism of ring spectra $\sigma: E\to F$, if $E$ satisfies the conditions of theorem \ref{susMain}, then so does $F$: The classes are given by $\sigma(c(n)_j)$. This leads to another corollary.
\begin{corollary}
	With the setting above, we have an isomorphism
	\[ F^{*,*}(G_n)\cong F^{*,*}(S)\otimes_{E^{*,*}(S)} E^{*,*}(G_n) \]
\end{corollary}

To prove Theorem \ref{susMain}, we need a simple lemma.
\begin{lemma}
	Let $f: X \to Y$ be a morphism of smooth schemes and let $g_1,\ldots,g_n$ be some homogeneous elements of $ E^{*,*}(X) $. Suppose that there exist a finite open covering $ Y=\bigcup U_i $ such that for any $ i $ and any open $ V\subset U_i $, the ring $ E^{*,*}(f^{-1}(V))$ is a free $ E^{*,*}(V)$-module with basis $ g_k|_{f^{-1}(V)} $. Then $ E^{*,*}(X) $ is a free $ E^{*,*}(Y)$-module with basis $\{ g_k \} $
\end{lemma}
\begin{proof}
	Applying the Mayer-Vietoris long exact sequence, we obtain a commutative diagram
\[\begin{tikzcd}
	{E^{i,*}(V_1\cup V_2)\otimes \bZ^{n_i}} & {(E^{i,*}(V_1)\oplus E^{i,*}(V_2))\otimes \bZ^{n_i}} & {E^{i,*}(V_1\cap V_2)\otimes \bZ^{n_i}}\\
	{E^{i,*}(f^{-1}(V_1\cup V_2))} & {E^{i,*}(f^{-1}(V_1))\oplus E^{i,*}(f^{-1}(V_2))} & {E^{i,*}(f^{-1}(V_1\cap V_2))}
	\arrow[from=1-1, to=1-2]
	\arrow[from=1-2, to=1-3]
	\arrow["\alpha"', from=1-1, to=2-1]
	\arrow["\cong"', from=1-2, to=2-2]
	\arrow["\cong", from=1-3, to=2-3]
	\arrow[from=2-2, to=2-3]
	\arrow[from=2-1, to=2-2]
\end{tikzcd}\]
where the vertical arrows are obtained by making use of the basis $g_k$. As both rows are exact, $\alpha$ is an isomorphism by the $5$-lemma. Note that the $g_k$ gives the sections of $ R^if_*E$.
\end{proof}
\begin{remark}
	This lemma actually gives a condition for the Leray spectral sequence to collapse at the $E_2$-page. We will see more of this perspective in the next part of the article.
\end{remark}

\begin{proof}[Proof of Theorem \ref{susMain}]
	We prove it by induction on $n$. For $n=1$, this can be concluded by Example \ref{caseAn} since $ G_1\cong \mathbb{A}^k_S\nonZero $.
	
	For general $n$ we apply the lemma to $f_n$, let $\{U_k\}$ be the open covering which trivialize $f_n$, then by our assumption for any open $ V \subset U_k, f_n^{-1}(V)=G_{n-1}\times_S V$ and $ c(n)_j|_{f_n^{-1}(V)} $ coincides with  $  p_2^*i^*(c(n)_j) $ in $E^{*,*}(f_n^{-1}(V))$. Using Assumption 2 and the induction hypothesis, we get that $ (p_2^*i^*(c(n)_1),\ldots,p_2^*i^*(c(n)_{n-1})) \sim (p_2^*(c(n-1)_1),\ldots,p_2^*(c(n-1)_{n-1})) $ forms generators of $ E^{*,*}(f_n^{-1}(V))\cong E^{*,*}(V)\otimes_{E^{*,*}(S)} E^{*,*}(G_{n-1})  $. This implies that  $ E^{*,*}(G_{n})  $ is a free $  E^{*,*}(\af^{kn}_S \nonZero)  $ square-zero algebra with generators $  (c(n)_1,\ldots,c(n)_{n-1}) $, i.e. the basis is given by $ c(n)_{i_1}\cup\ldots\cup c(n)_{i_j}, 1\leq i_1<\ldots<i_j \leq n-1 $.
	
	Then, combining with Example \ref{caseAn}, we can conclude that $ E^{*,*}(G_n) $ as a module is isomorphic to a free $  E^{*,*}(S)  $ square-zero algebra with generators $  (c(n)_1,\ldots,c(n)_{n-1},f_n^*(\theta_{kn}))\sim (c(n)_1,\ldots,c(n)_n) $.
\end{proof}

\subsection{Application to Sp-oriented cohomology}
Thanks to the corollary on the base change formula, we omit from now on the subscript $S$, at least when it is clear from the context. For instance, we write $ \afnz{n}:=\af^{n}_S \nonZero $.

Recall that for any $\Sp$-oriented ring spectrum $E\in \SH(S)$, we can define Borel classes $b_j\in E^{4j,2j}(B\Sp_{2n}), j\leq n$. Using loop spaces, we can define a new series of classes $ \beta(n)_j\in E^{4j-1,2j}(\Sp_{2n}) $:
\[ \beta(n)_j :  \Sp_{2n} \cong \Omega B\Sp_{2n}  \xr{\Omega b_j } \Sigma^{4j-1,2j} E.\]

Note further that $\Sp_{2n}$ is a $2$-fiber family since we have (locally split) fiber sequences for any $n\in \mathbb{N}$:
\[
\Sp_{2n-2} \hookrightarrow \Sp_{2n} \xr{p_n} \afnz{2n}.
\]

\begin{proposition}
	For any $\Sp$-oriented ring spectrum $E$, $\Sp_{2n}$ and $ \beta(n)_j $ satisfy the conditions of Theorem \ref{susMain}.
\end{proposition}
\begin{proof}
	We just need to prove the result for the universal $\Sp$-oriented spectrum $\MSp$. Condition 2 is automatically satisfied by the properties of Borel classes. We only need to verify that $\beta(n)_n$ satisfies Condition 1. For this, we consider the morphism $g$ defined as the looping of the canonical map
\[
	B\Sp_{2n}\to \mathrm{Th}(\mathcal{U}_{2n})
\]  
	where $\mathrm{Th}(\mathcal{U}_{2n})$ is the Thom space of the tautological symplectic bundle $\mathcal{U}_{2n}$ over $ B\Sp_{2n}$, and the morphism $f$  induced by the following diagram,  
	where $\mathcal{U}_{2n}^* = \mathcal{U}_{2n} \setminus B\Sp_{2n}$ is the complement of the zero section and $i$  is the fiber embedding along $S=\Sp_0\to \Sp_{2n}$
	\[\begin{tikzcd}
		{\afnz{2n}} & {\af^{2n}} \\
		{\mathcal{U}_{2n}^*} & {\mathcal{U}_{2n}} \\
		{*} & {\mathrm{Th}(\mathcal{U}_{2n})}
		\arrow[from=1-1, to=1-2]
		\arrow[from=2-1, to=2-2]
		\arrow[from=2-1, to=3-1]
		\arrow[from=3-1, to=3-2]
		\arrow[from=2-2, to=3-2]
		\arrow["\lrcorner"{anchor=center, pos=0.125, rotate=180}, draw=none, from=3-2, to=2-1]
		\arrow["\lrcorner"{anchor=center, pos=0.125}, draw=none, from=1-1, to=2-2]
		\arrow["i", from=1-1, to=2-1]
		\arrow[from=1-2, to=2-2]
	\end{tikzcd}\]

To prove the statement, we need to show that the diagram	
	\[\begin{tikzcd}
		{\Sp_{2n}} & {\afnz{2n}} \\
		& {\Omega\mathrm{Th}(\mathcal{U}_{2n})} \\
		& {\Sigma^{4n-1, 2n}\MSp}
		\arrow["f", from=1-2, to=2-2]
		\arrow["p_n", from=1-1, to=1-2]
		\arrow["g"', from=1-1, to=2-2]
		\arrow[from=2-2, to=3-2]
		\arrow["{\beta(n)_n}"', from=1-1, to=3-2]
		\arrow["{\theta_{2n}}", curve={height=-30pt}, from=1-2, to=3-2]
	\end{tikzcd}\]
	is commutative, which amounts to show that $g=fp_n$.

  This comes from the following commutative diagram, where $z$ is the zero section and the fact that $rh$ coincides with $i$ in the diagram above (this is due to the properties of canonical bundles over quaternionic grassmannians \cite[Theorem 4.1]{panin2022quaternionic}). 
	\[\begin{tikzcd}
		{\Sp_{2n}} & {\afnz{2n}} & {*} \\
		{*} & {\BSp_{2n-2}} & {\BSp_{2n}} \\
		& {\mathcal{U}_{2n}^*} & {\mathcal{U}_{2n}} \\
		& {*} & {\mathrm{Th}(\mathcal{U}_{2n})}
		\arrow[from=1-2, to=1-3]
		\arrow["h",from=1-2, to=2-2]
		\arrow[from=1-3, to=2-3]
		\arrow[from=2-2, to=2-3]
		\arrow["r"',"\cong",from=2-2, to=3-2]
		\arrow["z"',"\cong",from=2-3, to=3-3]
		\arrow[from=3-2, to=3-3]
		\arrow[from=3-2, to=4-2]
		\arrow[from=4-2, to=4-3]
		\arrow[from=3-3, to=4-3]
		\arrow["\lrcorner"{anchor=center, pos=0.125}, draw=none, from=1-2, to=2-3]
		\arrow["\lrcorner"{anchor=center, pos=0.125, rotate=180}, draw=none, from=4-3, to=3-2]
		\arrow["p_n",from=1-1, to=1-2]
		\arrow[from=1-1, to=2-1]
		\arrow[from=2-1, to=2-2]
		\arrow["\lrcorner"{anchor=center, pos=0.125}, draw=none, from=1-1, to=2-2]
	\end{tikzcd}\]
\end{proof}

\begin{theorem}\label{SpHmw}
	The graded ring $\HH^{*,*}(\Sp_{2n}) $ is a free $\HH^{*,*}(S)$-square-zero algebra generated by the classes $ \beta(n)_j\in \HH^{4j-1,2j}(\Sp_{2n}) $ for $j=1,\ldots,n$.
\end{theorem}
\begin{proof}
	We note that $\HH^{*,*}(\Sp_{2n})$ is $\Sp$-oriented, and combining this with our previous results yields the isomorphism for the additive structure. As for the multiplicative structure, we will come back to it later in Theorem \ref{VSpHmw}.
\end{proof}

We now establish the cofiber sequences required to obtain a motivic decomposition of $\Sp_{2n}$.

\begin{lemma}\label{SpCof}
	For any $n\geq 2$, there exists a cofiber sequence in $\mathcal{H}(S)$
	\[ (\Sp_{2n-2})_+\to(\Sp_{2n})_+\to (\Sp_{2n-2})_+\wedge (\afnz{2n})  \]
\end{lemma}
\begin{proof}
	We provide a proof of this lemma in Appendix \ref{appendix} using matrix computations. We also refer to Section \ref{SVCof}, where we will see similar computations.
\end{proof}

With this sequence in place, we can now derive the required decomposition result.

\begin{theorem}
	\label{SpMain}
	The cofiber sequence of Lemma \ref{SpCof} splits in $\DMt(S)$. In particular, it induces an isomorphism
	\[ \Mt(\Sp_{2n})\cong  \Mt(\afnz{2n})\otimes \Mt(\Sp_{2n-2})\]
	Moreover, the following decomposition holds in $\DMt(S)$: 
	\[\Mt(\Sp_{2n})\cong \Mt(\afnz{2n})\otimes \Mt(\afnz{2n-2})\otimes\ldots \otimes \Mt(\afnz{2}) \cong \bigoplus_{1\leq i_1<\ldots<i_j \leq n} \tbZ(d(i_1,\ldots,i_j))  .\] 
	where $\tbZ(d(i_1,\ldots,i_j))$ stands for $\tbZ(\sum_{l=1}^j 2i_l)[4(\sum_{l=1}^j i_l) -j] $.
\end{theorem}
\begin{proof}
	We will prove this theorem by induction on $n$. For the base case $n=1$, we note that $\Sp_2 \cong \afnz{2}$, and that $\Mt(\afnz{2})\cong \tbZ \oplus \tbZ(2)[3]$.

	Now, suppose the theorem holds for some $n-1\geq 1$. If the sequence in Lemma \ref{SpCof} splits after taking $\Mt$, we then have 
	\[ \Mt(\Sp_{2n})\cong \Mt(\Sp_{2n-2}) \oplus (\Mt(\Sp_{2n-2})\otimes \Mt(\afnz{2n})/\tbZ) \]
	\[\cong \bigoplus_{1\leq i_1<\ldots<i_j \leq n-1} \tbZ(d(i_1,\ldots,i_j)) \oplus  \bigoplus_{1\leq i_1<\ldots<i_j \leq n-1} \tbZ(d(i_1,\ldots,i_j,n)) \]

	To complete the proof, we are thus reduced to show that the triangle
	\[ \Mt(\Sp_{2n-2})\xr{i} \Mt(\Sp_{2n}) \to \Mt(\Sp_{2n-2})\otimes \Mt(\afnz{2n})/\tbZ \xr{+1} \] splits via a section $s:\Mt(\Sp_{2n})\to \Mt(\Sp_{2n-2})$ such that $si = \mrm{id}$.

	As $ \Mt(\Sp_{2n-2})\cong \bigoplus_{1\leq i_1<\ldots<i_j \leq n-1} \tbZ(d(i_1,\ldots,i_j)) $, we get the following diagram:
	\[
	\begin{tikzcd}
		\rHom_{\DMt}(\Mt(\Sp_{2n}),\Mt(\Sp_{2n-2}))\ar[r, "i^*"]\ar[d,"p_{i_1,\ldots,i_j}"] 
		&\rHom_{\DMt}(\Mt(\Sp_{2n-2}),\Mt(\Sp_{2n-2}))\ar[d,"p_{i_1,\ldots,i_j}"] \\
		 \rHom_{\DMt}(\Mt(\Sp_{2n}),\tbZ(d(i_1,\ldots,i_j)))\ar[r, "i^*"]\ar[d] 
		&\rHom_{\DMt}(\Mt(\Sp_{2n-2}),\tbZ(d(i_1,\ldots,i_j))\ar[d]\\
		\HH^{d(i_1,\ldots,i_j)}(\Sp_{2n},\bZ)\ar[r, "i^*"]
		&\HH^{d(i_1,\ldots,i_j)}(\Sp_{2n-2},\bZ)
	\end{tikzcd}
	\]
	where $p_{i_1,\ldots,i_j}$ is the projector onto the direct summand of $\Mt(\Sp_{2n-2})$ corresponding to $(i_1,\ldots,i_j)$. Since we have $i^*(\beta(n)_j)=\beta(n-1)_j$ for $1\leq j\leq n-1$, we can deduce that $i^*$ is surjective using Corollary \ref{SpHmw} on the bottom row. Thus the first line of the diagram is also surjective. Therefore, we can define such an $s:\Mt(\Sp_{2n})\to \Mt(\Sp_{2n-2})$ by $s=\oplus_{1\leq j_1\leq\ldots\leq j_k\leq n-1} \beta(n)_{j_1}\cup \ldots \cup \beta(n)_{j_k} $ and the theorem is proven.
\end{proof}
\begin{remark}
	Since $\rHom_{\mathrm{Mod}_E(\SH(S))}(\Sigma^{\infty}_+X,\Sigma^{p,q}E) \cong E^{p,q}(X)$, the same proof also works for the decomposition of $E\wedge \Sigma^{\infty}_+ \Sp_{2n}\in \mathrm{Mod}_E(\SH(S))$ for any $\Sp$-oriented spectrum $E$.
\end{remark}
\begin{remark}
	The results of this section also apply to other oriented theories. For example, for $\GL$-oriented theories, in particular $ \mathrm{H}_{\mathrm{M}}$, we can also get similar results for $ \mathrm{H}_{\mathrm{M}}(\GL_n) $ and $M(\GL_n)\in \DM(S) $ with Chern classes, etc.
\end{remark}

\section[Computations on Stiefel varieties]{Computations on Stiefel varieties, $\GL_n$ and $\SL_n$}

As we mentioned before, $\HH$ is not a $\GL$-oriented cohomology theory, making the methods used for computing the ordinary motives of algebraic groups and Stiefel varieties impossible to apply. Instead, we will follow the classic computations of the singular cohomology groups of $O_n$, which make use of the Leray spectral sequences and inductive computation involving all Stiefel manifolds.

For $k\leq n$, let then $V_k(\af^n)$ be the Stiefel variety, which is the variety of $k$ linearly independent vectors in $\af^n_K$. In particular, $V_n(\af^n_K)\cong \GL_n$ and $V_1(\af^n_K)\cong \af^{n}_K\setminus \{0\}$. We can also consider these varieties over a more general base scheme $S$ over $K$, in which case the relevant varieties are just pulled-back from $K$. In the sequel, we omit the subscript with the convention that we always work over a base scheme $S$. There is a canonical identification
\[
V_k(\af^n)\cong \mathrm{GL}_n/\mathrm{GL}_{n-k}
\]
which allows us to see Stiefel varieties as homogeneous spaces.

We first notice that for $l<k\leq n$ there are fiber sequences for Stiefel varieties of the form
\begin{equation}
	\label{fibSV}
V_{k-l}(\af^{n-l})\to V_k(\af^n)\to V_l(\af^n)
\end{equation}
as one can see from \cite[Theorem 2.1]{Asok12a}. In fact, these fiber sequences are Zariski locally split sequences.

Let $\mathbf{HS}_{2k}:=\Meta(S)\oplus \Meta(S)(2k)[4k-1]$ and $\mathbf{HS}_{2k+1}:=\Meta(S)\oplus \Meta(S)(4k+1)[8k]$. As explained in the introduction, we might as well have set $\mathbf{HS}_{2k}\cong\Meta(S)\oplus \Meta(S)[2k-1]$ and $\mathbf{HS}_{2k+1}\cong\Meta(S)\oplus \Meta(S)[4k-1]$. Let $\Heta$ be the $\eta$-inverted cohomology theory induced by the $\eta$-inverted cycle module $\mathbf{W}$ (see more on section \ref{MWmod}) over $K$. Then $\Heta$ can be treated as a one-graded theory: $ \Heta^*:=\oplus_n \Heta^{n+*,n}\cong \Heta^{*,0}[\eta,\eta\inv]$.

\begin{theorem}
	\label{SVCohMain}
	The following computations hold:
\[
\Heta^*(V_{k}(\af^{n}))\cong \begin{cases}\Heta^*(V_{k-1}(\af^{n})) \otimes_{\Heta^*(S)} \Heta^*(\afnz{n-k+1}) & \text{ if $ n-k $ is even.} \\
\Heta^*(V_{k-2}(\af^{n})) \otimes_{\Heta^*(S)} \Heta^*(\mathbf{HS}_{n-k+2}) & \text{ if $ n-k $ is odd. } \end{cases}
\]
Consequently, we inductively obtain an isomorphism of graded algebras 
	\[\Heta^{*}(V_{k}(\af^{n}))\cong \Heta^{*}(S)[G]/(g^2-\delta_{\gamma_0}(g)[-1]g , g\in G)\] 
	where the generators $G$ are given by 
	\begin{equation*}
		G=\begin{cases}
			G_1\sqcup G_2\sqcup G_3 &\text{$n$ even, $n-k$ even,}\\
			G_2\sqcup G_3 &\text{$n$ odd, $n-k$ even,}\\
			G_1\sqcup G_2 &\text{$n$ even, $n-k$ odd,}\\
			G_2 &\text{$n$ odd, $n-k$ odd,}\\
			\end{cases}
	\end{equation*}
	with 
	\begin{eqnarray*}
		G_1&=&\{\alpha_{n-1}\in \Heta^{n-1}(\mathbf{HS}_n)\},\\
		G_2&=&\{\beta_{4j-1}\in \Heta^{4j-1}(\mathbf{HS}_{2j+1})|n-k<2j<n\},\\
		G_3&=&\{\gamma_{n-k} \in \Heta^{n-k}(\afnz{n-k+1})\},
	\end{eqnarray*}
	and $\delta_{\gamma_0}(\gamma_0)=1$, otherwise $ \delta_{\gamma_0}(g)=0$.

\end{theorem}

We can also give the motivic decomposition of $\Meta(V_k(\af^n))$ in $\DMt(S)[\eta\inv]$ for any $k\leq n$. We note that in this category $\tbZ[\eta^{-1}]\cong \mathbf{W}_*$, the \MW cycle module $\mathbf{W}_*:= W[\eta,\eta\inv]$ where $\eta$ is of degree $-1$, i.e. $\mathbf{W}_i := \eta^{-i} \mathbf{W}_0$. Therefore, we have 
\[
\mathbf{HS}_{2k}= \mathbf{W}_0 \oplus \mathbf{W}_{2k}[2k-1],\ \mathbf{HS}_{2k+1}= \mathbf{W}_0 \oplus \mathbf{W}_{4k+1}[4k-1].
\]

\begin{theorem}
	\label{SVMain}
	In $\DMt(S)[\eta\inv]$, we have the following motivic decompositions
	\[\Meta(V_k(\af^{2i})) \cong \mathbf{HS}_{2i}\otimes \Meta(V_{k-1}(\af^{2i-1}))\] 
	and 
	\[\Meta(V_k(\af^{2i+1})) \cong \mathbf{HS}_{2i+1}\otimes \Meta(V_{k-2}(\af^{2i-1})) .\]
	Consequently, we obtain the following full decompositions by induction:
	\[\Meta(V_{2j}(\af^{2i})) \cong \mathbf{HS}_{2i} \otimes  \mathbf{HS}_{2i-1}\otimes \mathbf{HS}_{2i-3}\otimes \ldots \otimes \mathbf{HS}_{2i-1-2(j-2)} \otimes \Meta(\afnz{2i+1-2j}) \]
	\[\Meta(V_{2j+1}(\af^{2i})) \cong \mathbf{HS}_{2i} \otimes \mathbf{HS}_{2i-1}\otimes \mathbf{HS}_{2i-3}\otimes \ldots \otimes \mathbf{HS}_{2i-1-2(j-1)}  \]
	\[\Meta(V_{2j}(\af^{2i+1})) \cong \mathbf{HS}_{2i+1}\otimes \mathbf{HS}_{2i-1}\otimes \ldots \otimes \mathbf{HS}_{2i+1-2(j-1)}  \]
	\[\Meta(V_{2j+1}(\af^{2i+1})) \cong \mathbf{HS}_{2i+1}\otimes \mathbf{HS}_{2i-1}\otimes \ldots \otimes \mathbf{HS}_{2i+1-2(j-1)} \otimes \Meta(\afnz{2i+1-2j})  \]
\end{theorem}

It is important to note that the generators in Theorem \ref{SVCohMain} correspond to the decompositions in Theorem \ref{SVMain}.

\begin{corollary}
	In $\DMt(S)[\eta\inv]$, we have 
	\[\Meta(\GL_{2i})\cong \mathbf{HS}_{2i} \otimes  \mathbf{HS}_{2i-1}\otimes \mathbf{HS}_{2i-3}\otimes \ldots \otimes \mathbf{HS}_{3} \otimes \Meta(\Gm)\]
	\[\Meta(\GL_{2i+1}) \cong \mathbf{HS}_{2i+1}\otimes \mathbf{HS}_{2i-1}\otimes \ldots \otimes \mathbf{HS}_{3} \otimes \Meta(\Gm) \]
	and 
	\[\Meta(\SL_{2i})\cong \mathbf{HS}_{2i} \otimes  \mathbf{HS}_{2i-1}\otimes \mathbf{HS}_{2i-3}\otimes \ldots \otimes \mathbf{HS}_{3} \]
	\[\Meta(\SL_{2i+1}) \cong \mathbf{HS}_{2i+1}\otimes \mathbf{HS}_{2i-1}\otimes \ldots \otimes \mathbf{HS}_{3}. \]
\end{corollary}
\begin{proof}
	For the case of $\SL_n$, we just need to show that $\SL_n\cong V_{n-1}(\af^n)$ which follows from the fact that $V_{n-1}(\af^n)\cong \GL_n/\GL_1$.	
\end{proof}


Let us stress once again that our focus in this section is only on the additive structure; we will study the multiplicative structure in section \ref{ring}. 

Let $\mrm{PH}(k,n)$ (resp. $\mrm{P}(k,n)$) denote the property that Theorem \ref{SVCohMain} (resp. Theorem \ref{SVMain}) holds for $V_{k}(\af^n)$ (for the additive structure). Our strategy is to use induction to prove that $\mrm{PH}(k,n)$ holds, and deduce that $\mrm{P}(k,n)$ also holds. We notice that for $k=1$ everything is clear from the results on motivic spheres, and we proceed with the case $k=2$ in the next section.

\subsection[A cofiber sequence for V2]{A cofiber sequence for $V_2(\af^n)$.}
\label{secV2}
The points of $V_2(\af^n)$ can be expressed as
$\begin{bmatrix}
a_1 & \cdots & a_n\\
b_1 & \cdots & b_n
\end{bmatrix}$ 
(with the condition on the rank of the matrix). For such a matrix, we write $a_i*$ or $(a_i\cdots a_j)*$ if $(a_i, \cdots ,a_j)\in \afnz{j-i+1}$ and $a_i\circ $ or $(a_i\cdots a_j)\circ$ if $(a_i, \cdots ,a_j)\in \af^{j-i+1}$ without extra condition.

Next we consider the open subvariety $U\subset V_2(\af^n)$, which is defined by $(a_1, \cdots ,a_{n-1})\in \afnz{n-1}$, that is
$$\begin{bmatrix}
(a_1 & \cdots & a_{n-1})* & a_n\circ\\
b_1 & \cdots &\cdots & b_n
\end{bmatrix}\in U .$$
And let $ W \subset U$ be a closed subvariety defined by $a_n=1, b_n=0$. We now show that $W$ is a retract of $U$ in $\mathcal{H}(S)$

\begin{lemma}\label{shrinkOpen}
	There is a retraction $r: U \to W \cong \afnz{n-1}\times \afnz{n-1}$ in $\mathcal{H}(S)$ given by a composition of elementary transformations. In particular, $r$ induces an isomorphism $U\cong \afnz{n-1}\times \afnz{n-1}$ up to $\af^1$-homotopy.
\end{lemma}
\begin{proof}
	It is convenient to replace $\afnz{n-1}$ with $Q_{2n-3}$, whose points consist of pairs $(\mathbf{a},\mathbf{x})$ such that $\mathbf{a} \mathbf{x}^T=1$, while the projection $Q_{2n-3}\to \afnz{n-1}$ is given by $(\mathbf{a},\mathbf{x})\mapsto \mathbf{a}$. Note that the projection is an $\af^1$-equivalence, since it is locally trivial with affine fibers. The pull-back $U'$ of $Q_{2n-3}$ along the projection $U\to \afnz{n-1}$ is also $\af^1$-weak equivalent to $U$ and we may then consider the elements in $U'$ as pairs $(\mathbf{x},\begin{bmatrix}
		\mathbf{a}*  & a_n\circ\\
		  \mathbf{b} & b_n
	\end{bmatrix})$ with $\mathbf{a} \mathbf{x}^T=1$.

Using the following elementary transformations (which are $\af^1$-weak equivalences), we can change $a_n$ to $1$: $$(\mathbf{x},\begin{bmatrix}
		\mathbf{a}*  & a_n\circ\\
		  \mathbf{b} & b_n
	\end{bmatrix})
	\xr{\times\begin{bmatrix}
		\mathbf{I}& -(a_n-1)\mathbf{x}^T\\
		0&1
	\end{bmatrix}}
	(\mathbf{x},\begin{bmatrix}
		\mathbf{a}*  & 1\\
		  \mathbf{b} & b_n-(a_n-1)\mathbf{b} \mathbf{x}^T
	\end{bmatrix}).$$
Now, we may consider the following transformation $$(\mathbf{x},\begin{bmatrix}
		\mathbf{a}*  & 1\\
		  \mathbf{b} & b_n-(a_n-1)\mathbf{b} \mathbf{x}^T
	\end{bmatrix})
	\xr{\begin{bmatrix}
		1& 0\\
		-(b_n-(a_n-1)\mathbf{b} \mathbf{x}^T) &1
	\end{bmatrix}\times}
	(\mathbf{x},\begin{bmatrix}
		\mathbf{a}* & 1\\
		 \mathbf{b}-(b_n-(a_n-1)\mathbf{b} \mathbf{x}^T)\mathbf{a} & 0
	\end{bmatrix}).$$
Note that indeed $\mathbf{b}-(b_n-(a_n-1)\mathbf{b} \mathbf{x}^T)\mathbf{a}\in \afnz{n-1}$ as the matrix is of rank $2$. Finally, we conclude by applying Lemma \ref{Mretract}.

\end{proof}
The elements in the complement of $U$ are simply given by $\begin{bmatrix}
	& 0 & & a_n*\\
	(b_1 & \cdots & b_{n-1})* & b_n\circ
	\end{bmatrix}$. As a result, the complement is isomorphic to $\Gm \times \afnz{n-1}$ up to homotopy, and its normal bundle is trivial, with an obvious choice of trivialization. Thus, we obtain the following cofiber sequence: 
\begin{equation} \label{eq:cofV2}
	(\afnz{n-1}\times \afnz{n-1})_+\to V_2(\af^n)_+\to (\af^{n-1}/\afnz{n-1})\wedge (\Gm \times \afnz{n-1})_+ \xr{\alpha }
\end{equation}

We now consider the following diagram
\[\begin{tikzcd}
	{(\afnz{n-1}\times \afnz{n-1})_+} & {(\af^{n-1}\times \afnz{n-1})_+} & {(\af^{n-1}/\afnz{n-1})\wedge \afnz{n-1}_+} \\
	{(\afnz{n-1}\times \afnz{n-1})_+} & {V_2(\af^n)_+} & {(\af^{n-1}/\afnz{n-1})\wedge (\Gm \times \afnz{n-1})_+} \\
	{(\af^{n-1}\times \afnz{n-1})_+} & {V_2(\af^n)_+} & {C}
	\arrow[Rightarrow, no head, from=1-1, to=2-1]
	\arrow[from=2-1, to=3-1]
	\arrow[from=3-1, to=3-2]
	\arrow[from=2-1, to=2-2]
	\arrow[from=1-1, to=1-2]
	\arrow[from=1-2, to=1-3]
	\arrow[from=2-2, to=2-3]
	\arrow[Rightarrow, no head, from=2-2, to=3-2]
	\arrow[from=3-2, to=3-3]
	\arrow[dashed, from=2-3, to=3-3]
	\arrow[dashed, from=1-3, to=2-3]
	\arrow[from=1-2, to=2-2]
\end{tikzcd}\]
in which the top cofiber sequence is given by the obvious morphism 
\[
\afnz{n-1}\times \afnz{n-1}\to \af^{n-1}\times \afnz{n-1}.
\] 
and the bottom cofiber sequence is given by the morphism
\[ 
  \af^{n-1}\times \afnz{n-1} \to V_2{\af^n}, (\mathbf{a}\circ, \mathbf{b}*) \mapsto  \begin{bmatrix}
    \mathbf{a}\circ & 1 \\ 
    \mathbf{b}* & 0
  \end{bmatrix}
\]
The octahedron axiom shows that the right-hand column is a cofiber sequence, in which the top vertical map is induced by the map $\afnz{n-1}\to\Gm \times \afnz{n-1}$ given by $\mathbf{a}\mapsto  (1,\mathbf{a})$. 
Since $(\af^{n-1}/\afnz{n-1})\wedge \Gm\cong \afnz{n}$, we obtain that $C\cong  \afnz{n} \wedge \afnz{n-1}_+$ and the bottom cofiber sequence becomes:

\[\afnz{n-1}_+ \to {V_2(\af^n)}_+ \to \afnz{n} \wedge \afnz{n-1}_+ \xr{\beta }\]

Our goal is to comprehend the connecting map $\beta$. To achieve our goal, we go back to the cofiber sequence (\ref{eq:cofV2}), which shows that $\beta$ originates from the relevant direct summand of $\alpha$. To make the latter explicit, we refer to the diagram below:

{\small\[ 
	\begin{tikzcd}[column sep=small]
	{(\afnz{n-1}\times \Gm\times \afnz{n-1})_+} & {(\af^{n-1}\times \Gm\times \afnz{n-1})_+} & {(\af^{n-1}/\afnz{n-1})\wedge (\Gm \times \afnz{n-1})_+} & { } \\
	U_+ & {V_2(\af^n)_+} & {(\af^{n-1}/\afnz{n-1})\wedge (\Gm \times \afnz{n-1})_+} & { } \\
	{(\afnz{n-1}\times \afnz{n-1})_+}
	\arrow[from=3-1, to=2-2]
	\arrow[from=2-2, to=2-3]
	\arrow["\alpha", from=2-3, to=2-4]
	\arrow["j_2", from=1-2, to=2-2]
	\arrow[Rightarrow, no head, from=1-3, to=2-3]
	\arrow[from=1-1, to=1-2]
	\arrow[from=1-2, to=1-3]
	\arrow[from=2-1, to=2-2]
	\arrow["r","\cong"', from=2-1, to=3-1]
	\arrow["j_1", from=1-1, to=2-1]
	\arrow["i", from=1-3, to=1-4]
	\arrow["\gamma_+"', curve={height=18pt}, from=1-1, to=3-1]
\end{tikzcd}\]}
in which the morphisms $j_1$ and $j_2$ are defined by the formulas
$$ j_1:(\mathbf{a},a_n,\mathbf{b})\mapsto 
\begin{bmatrix} 
	 (a_1 & \cdots & a_{n-1})* & a_n*\\ 		
	(b_1 & \cdots &b_{n-1})* & 0 	
\end{bmatrix},\phantom{i}
j_2: (\mathbf{a},a_n,\mathbf{b})\mapsto 
\begin{bmatrix} 
	(a_1 & \cdots & a_{n-1})\circ & a_n*\\ 		
	(b_1 & \cdots &b_{n-1})* & 0 	\end{bmatrix}.$$
Therefore, we have $\alpha= \Sigma \gamma_+ \circ i$, where $i$ denotes the canonical Gysin map induced by the connecting map $(\af^{n-1}/\afnz{n-1})\to \Sigma \afnz{n-1}$ and $\gamma= r\circ j_1$. Indeed, we have the following lemma.
\begin{lemma}
Let 
\[
\mathcal{Z}_1\xr{f} \mathcal{Z}_2\xr{g} \mathcal{Z}_3\xr{i}\Sigma\mathcal{Z}_1
\]
be a cofiber sequence, and let $\mathcal X$ be a space. Then, the following sequence 
\[
(\mathcal{Z}_1\times\mathcal X)_+ \xr{f\times 1} (\mathcal{Z}_2\times\mathcal X)_+\xr{p_{\mathcal{Z}_3}(g\times 1)} \mathcal{Z}_3\wedge \mathcal X_+\xr{i\wedge 1}\Sigma\mathcal{Z}_1\wedge \mathcal X_+,
\]
in which $p_{\mathcal{Z}_3}:(\mathcal{Z}_3\times \mathcal X)_+\to \mathcal{Z}_3\wedge \mathcal X_+$ is the canonical projection, is a cofiber sequence.
\end{lemma}

\begin{proof}
We have a cofiber sequence
\[
\mathcal{Z}_1\wedge\mathcal X_+ \xr{f\wedge 1} \mathcal{Z}_2\wedge\mathcal X_+\xr{g\wedge 1} \mathcal{Z}_3\wedge \mathcal X_+\xr{i\wedge 1}\Sigma\mathcal{Z}_1\wedge \mathcal X_+
\]
and notice that $(\mathcal{Z}_i \times \mathcal{X})_+ \cong ( \mathcal{Z}_i )_+ \wedge \mathcal{X}_+ $ we have the cofiber sequences
\[
  \mathcal{X}_+ \to (\mathcal{Z}_i \times \mathcal{X})_+ \xr{p_{\mathcal{Z}_i}} \mathcal{Z}_i \wedge \mathcal{X}_+
\]
Then there is a commutative diagram in which the rows and columns are cofiber sequences 
\[\begin{tikzcd}
	{\mathcal X_+} & {\mathcal X_+} & \star & {\Sigma( \mathcal X_+)} \\
	{(\mathcal{Z}_1\times\mathcal X)_+} & {(\mathcal{Z}_2\times\mathcal X)_+} & C & {\Sigma(\mathcal{Z}_1\times\mathcal X)_+} \\
	{\mathcal{Z}_1\wedge\mathcal X_+ } & {\mathcal{Z}_2\wedge\mathcal X_+} & {\mathcal{Z}_3\wedge \mathcal X_+} & {\Sigma\mathcal{Z}_1\wedge \mathcal X_+}
	\arrow[from=1-1, to=2-1]
	\arrow[Rightarrow, no head, from=1-1, to=1-2]
	\arrow[from=1-2, to=1-3]
	\arrow[from=1-2, to=2-2]
	\arrow["{f\times 1}", from=2-1, to=2-2]
	\arrow[from=2-2, to=2-3]
	\arrow[from=1-3, to=2-3]
	\arrow["{p_{\mathcal{Z}_1}}"', from=2-1, to=3-1]
	\arrow["{f\wedge 1}", from=3-1, to=3-2]
	\arrow["{p_{\mathcal{Z}_2}}"', from=2-2, to=3-2]
	\arrow["{g\wedge 1}", from=3-2, to=3-3]
	\arrow[dashed, from=2-3, to=3-3]
	\arrow[from=1-3, to=1-4]
	\arrow[from=2-3, to=2-4]
	\arrow["{i\wedge 1}", from=3-3, to=3-4]
	\arrow[from=1-4, to=2-4]
	\arrow["{\Sigma p_{\mathcal{Z}_1}}"', from=2-4, to=3-4]
\end{tikzcd}\]
The claim now follows from the fact that the dotted arrow is an equivalence, and from the equality $(g\wedge 1)p_{\mathcal{Z}_2}=p_{\mathcal{Z}_3}(g\times 1)$.
\end{proof}

We now provide the explicit expression of $\gamma$ after substituting $\afnz{n-1}$ with $Q_{2n-3}$.
\begin{lemma}
	The morphism $\gamma= r\circ j_1:Q_{2n-3}\times \Gm\times \afnz{n-1}\to Q_{2n-3}\times \afnz{n-1} $
	is defined by 
	\[ (\mathbf{x},\mathbf{a},a_n,\mathbf{b})\mapsto
	(\mathbf{x},\mathbf{a},\mathbf{b}\mathbf{M}(\mathbf{x},\mathbf{a},a_n)^T)
		\]
		where $\mathbf{M}(\mathbf{x},\mathbf{a},a_n)=\mathbf{I}+(a_n-1)\mathbf{a}^T\mathbf{x}$.
		Moreover, $\mathbf{M}(\mathbf{x},\mathbf{a},a_n)$ is invertible, and $\mathbf{a}$ is an eigenvector of the eigenvalue $a_n$. 
\end{lemma}
\begin{proof}
	This first part follows immediately from the definition of the retraction $r$ in Lemma \ref{shrinkOpen}.
	Now, let us consider the matrix $\begin{bmatrix}
		1 & -\mathbf{x}\\
		(a_n-1)\mathbf{a}^T& \mathbf{I}
	\end{bmatrix}$. Applying the elementary transformations given by the matrices 
  $\begin{bmatrix}
		1 & \mathbf{x}\\
		\mathbf{0} & \mathbf{I}
	\end{bmatrix}$ 
  and 
  $\begin{bmatrix}
		1 & \mathbf{0}\\
		-(a_n-1)\mathbf{a}^T& \mathbf{I}
	\end{bmatrix}$, we have 
  \[ 
  \begin{bmatrix}
		1 & \mathbf{x}\\
		\mathbf{0} & \mathbf{I}
	\end{bmatrix}
  \begin{bmatrix}
		1 & -\mathbf{x}\\
		(a_n-1)\mathbf{a}^T& \mathbf{I}
	\end{bmatrix}
  \begin{bmatrix}
		1 & \mathbf{0}\\
		-(a_n-1)\mathbf{a}^T& \mathbf{I}
	\end{bmatrix}
  =
  \begin{bmatrix}
		1+(a_n-1)\mathbf{a}\mathbf{x}^T=a_n & \mathbf{0}\\
		\mathbf{0}& \mathbf{I}
	\end{bmatrix}
  \]
  \[ 
  \begin{bmatrix}
		1 & \mathbf{0}\\
		-(a_n-1)\mathbf{a}^T& \mathbf{I}
	\end{bmatrix}
  \begin{bmatrix}
		1 & -\mathbf{x}\\
		(a_n-1)\mathbf{a}^T& \mathbf{I}
	\end{bmatrix}
  \begin{bmatrix}
		1 & \mathbf{x}\\
		\mathbf{0} & \mathbf{I}
	\end{bmatrix}
  =
  \begin{bmatrix}
		1 & \mathbf{0}\\
		\mathbf{0}& \mathbf{I}+(a_n-1)\mathbf{a}^T\mathbf{x}
	\end{bmatrix} .
  \]
	 Thus $\det(\mathbf{I}+(a_n-1)\mathbf{a}^T\mathbf{x})=a_n$, which is invertible. Finally, we have $(\mathbf{I}+(a_n-1)\mathbf{a}^T\mathbf{x})\mathbf{a}^T=\mathbf{a}^T+(a_n-1)\mathbf{a}^T=a_n\mathbf{a}^T$.
\end{proof}
\begin{corollary}
	In $\Da(K)$, the Gysin map $\beta$ can be expressed as 
\[
	\beta=\Sigma\gamma'_+ \circ i',
\]
where $i':(\af^{n-1}/\afnz{n-1})\wedge \Gm \wedge \afnz{n-1}_+ \to \Sigma \afnz{n-1} \wedge (\Gm)_+ \wedge \afnz{n-1}_+ $ is the canonical split injection, and 
\[
\gamma':Q_{2n-3}\times \Gm\times \afnz{n-1}\to \afnz{n-1}
\] 
is defined by $(\mathbf{x},\mathbf{a},a_n,\mathbf{b})\mapsto
	\mathbf{b}\mathbf{M}(\mathbf{x},\mathbf{a},a_n)^T$.
\end{corollary}
\begin{remark}
	The matrix $ \mathbf{M}$ actually almost coincides with the abstract map $\Gm\wedge \mathbb{P}^{n-1}_+\to \GL_n$ defined in \cite[\S 5]{williams2012motivic}.
	The morphism $\gamma'$ can be seen as the analogue of the holonomy data in the classical case.
\end{remark}
Now consider the morphism 
\[
d: (\afnz{n-1})\times \Gm \to (\afnz{n-1})\times \Gm\times (\afnz{n-1})
\] 
defined by $d(\mathbf{a},a_n)=(\mathbf{a},a_n,\mathbf{a})$, which induces a split injection 
\[
d: (\afnz{n-1}) \wedge (\Gm)_+ \to (\afnz{n-1}) \wedge (\Gm)_+ \wedge \afnz{n-1}_+ .
\] 
From above, it is easy to see that $\gamma'\circ d = \theta:(\afnz{n-1})\times \Gm\to \afnz{n-1}, (\mathbf{a},a_n)\mapsto a_n\mathbf{a}$. Combining with (\cite[Lemma 2.48]{morel2020cellular}), we obtain the following result.
\begin{proposition}
Let $\tau:  \afnz{n} \wedge \afnz{n-1}_+ \xr{\beta} \Sigma \afnz{n-1}_+ \to \Sigma \afnz{n-1}$.
	In $\Da(K)$, the composite morphism
	\[ \afnz{n}\cong \Sigma\afnz{n-1}\wedge \Gm\xr{\Sigma d} \afnz{n} \wedge (\afnz{n-1}_+) \xr{\tau } \Sigma \afnz{n-1}_+  \]
	is given by $n_\epsilon \eta$, i.e. it is trivial if $n$ is even, and equal to $\eta$ if $n$ is odd. Furthermore, $\tau$ is trivial on the summand $ \afnz{n} \wedge \afnz{n-1}$.
\end{proposition}
\begin{proof}
	As discussed above, we need to understand $\theta$, which can be done passing to the first nontrivial homotopy sheaf. We then obtain a map $\tau \circ \Sigma d = \Sigma\theta: \mathbf{K}^{\rMW}_{n-1}\otimes \mathbf{K}^{\rMW}_{1}[n] \to \mathbf{K}^{\rMW}_{n-1}[n] $ and the result follows from \cite[Lemma 2.48]{morel2020cellular}.

  Then we notice that $\tau$ is trivial on the summand $(\afnz{n}) \wedge  \afnz{n-1}$  for degree reasons. In fact for $ n-1 >1$, 
$$ \rHom_{\Da(K)}( (\afnz{n}) \wedge \afnz{n-1}, \Sigma\afnz{n-1}) \cong \rHom_{\Da(K)}( \Gm^{\wedge 2n-1} ,\Gm^{\wedge n-1}[-(n-2)]) $$ 
This is $0$ following from $\af^1$-connectivity theorem \cite[Theorem 6.22]{morel2012a1}.

When $n-1=1$, in this case, 
\[
  \gamma' : \Gm \times \Gm \times \Gm \to \Gm, (a, a_n, b) \mapsto a_nb
\]
only two factors are in the multiplying, thus it induces $0 \in \rHom_{\Da(K)}( \Gm^{\wedge 3} ,\Gm)$.
\end{proof}

\begin{remark}
	\label{etaSplit}
	For $n$ odd, this gives a non-canonical splitting: the composition with the projection to the summand
	\[  \afnz{n}\xr{\Sigma d} \afnz{n} \wedge \afnz{n-1}_+\cong \afnz{n} \wedge \afnz{n-1}\vee \afnz{n}\to \afnz{n} \]
	is the identity. Thus the complement(cone) of $\Sigma d$ is given by 
	\[C(\Sigma d)\cong \afnz{n} \wedge \afnz{n-1} \sqcup_{\Sigma d, \afnz{n}} \afnz{n} \cong \afnz{n} \wedge \afnz{n-1}.\]
\end{remark}

Recall that we defined $C_\eta(S)$ in the introduction to be the cone of $\Mt(S) \xr{\eta}  \Mt(S)(1)[1]$ and 
$$\wt{\mathbf{HS}}_{2k+1}:= \Mt(S)\oplus C_\eta(S)(2k)[4k] \oplus \Mt(S)(4k+1)[8k].$$ 

The following corollary justifies the conjecture \ref{IntegralConj} on very basic examples.

\begin{corollary}
	The following motivic decompositions hold in $\DMt(S)$:
	\[ \Mt(V_2(\af^{n}))\cong \Mt(\afnz{n})\otimes \Mt(\afnz{n-1}) \] for $n$ even, and
	\[ \Mt(V_2(\af^{n}))\cong  \wt{\mathbf{HS}}_{n}\]
	for $n$ odd.
\end{corollary}

\begin{corollary}
	In the $\eta$-inverted category $\Da(K)[\eta^{-1}]$, we have the  following split cofiber sequences:
 \[\afnz{n-1}_+ \to {V_2(\af^n)}_+ \to \afnz{n} \wedge (\afnz{n-1}_+) \xr{0 } \]
for $n$ even, and
\[S^0 \to {V_2(\af^n)}_+ \to \afnz{n} \wedge \afnz{n-1} \xr{0 } \]
for $n$ odd.
\end{corollary}

\begin{proof}
	The result is obvious if $n$ is even. In case $n$ is odd, we consider the following diagram:
\[\begin{tikzcd}
	{S^0} & {\afnz{n-1}_+} & {\afnz{n-1}} \\
	{S^0} & {{V_2(\af^n)}_+} & {C(\Sigma\inv \tau)} \\
	{\afnz{n-1}_+} & {{V_2(\af^n)}_+} & {\afnz{n} \wedge \afnz{n-1}_+ }
	\arrow[Rightarrow, no head, from=1-1, to=2-1]
	\arrow[from=2-1, to=3-1]
	\arrow[from=1-1, to=1-2]
	\arrow[from=2-1, to=2-2]
	\arrow[from=3-1, to=3-2]
	\arrow[from=1-2, to=1-3]
	\arrow[from=1-2, to=2-2]
	\arrow[Rightarrow, no head, from=2-2, to=3-2]
	\arrow[from=2-2, to=2-3]
	\arrow[from=3-2, to=3-3]
	\arrow[dashed, from=1-3, to=2-3]
	\arrow[dashed, from=2-3, to=3-3]
	\arrow["\ \ \tau"',"+1"{description}, curve={height=50pt}, dotted, from=3-3, to=1-3]
\end{tikzcd}\]
Since $\tau \circ \Sigma d =\eta$ is an isomorphism, we can deduce that   $C(\Sigma\inv\tau)\cong C(\Sigma d) \cong \afnz{n} \wedge \afnz{n-1}$ from Remark \ref{etaSplit}.
\end{proof}

The same outcome holds in $\DMt(S)[\eta\inv]$.
\begin{corollary}
The following decompositions hold in $\DMt(S)[\eta\inv]$
\[\Meta(V_2(\af^n))\cong \Meta(\afnz{n})\otimes \Meta(\afnz{n-1})\cong \mathbf{HS}_n\otimes \Meta(\afnz{n-1})\]
for $n$ even, and 
 \[\Meta(V_2(\af^n))\cong \mathbf{HS}_n\]
for $n$ odd. 
\end{corollary}
If we apply these results to the cohomology theory $\Heta$, this implies the decomposition of  $\Heta(V_{2}(\af^{n}))$ in terms of the cohomology of $\afnz{n}$ and $\afnz{n-1}$.
\begin{corollary}
	We have an isomorphism
	\[ 
	\Heta^{*}(V_{2}(\af^{n}))\cong \begin{cases}\Heta^{*}(\afnz{n}) \otimes_{\Heta^{*}(S)} \Heta^{*}(\afnz{n-1}) & \text{ for $n$ even.} \\
	\Heta^{*}(\mathbf{HS}_{n}) & \text{for $ n $ odd.}\end{cases}
	\] 
\end{corollary}
Note that the two corollaries show that both $\mrm{P}(2,n)$ and $\mrm{PH}(2,n)$ hold true for all $n$.

\subsection[Cofiber sequences for Vk]{Cofiber sequences for $V_k(\af^n)$}
\label{SVCof}

In this section, we show that computations analogous to the ones in Section \ref{secV2} can be applied to $V_k(\af^n)$ for $n\geq k\geq 2$.

Applying the same transformations
\[\begin{bmatrix}
	(a_1 & \cdots & a_{n-1})* & a_n\circ\\
	\ddots\\
	b_{j1} & \cdots &\cdots & b_{jn}\\
	\ddots
\end{bmatrix} \mapsto 
\begin{bmatrix}
	(a_1 & \cdots & a_{n-1})* & 1\\
	\ddots\\
	(b'_{j1} & \cdots & b'_{j,n-1})* & 0\\
	\ddots
\end{bmatrix}\]
and lemma \ref{Mretract}, we obtain the following cofiber sequences for $k\geq 2$:
\[V_{k-1}(\af^{n-1})_+ \to V_k(\af^n)_+ \to \afnz{n} \wedge V_{k-1}(\af^{n-1})_+ \xr{\beta }\]
and the corresponding diagrams:
{\small \[ 
	\begin{tikzcd}[column sep=tiny]
	{(\afnz{n-1}\times \Gm\times V_{k-1}(\af^{n-1}))_+} & {(\af^{n-1}\times \Gm\times V_{k-1}(\af^{n-1}))_+} & {\af^{n-1}/\afnz{n-1}\wedge (\Gm \times V_{k-1}(\af^{n-1}))_+} & { } \\
	U & {V_k(\af^n)_+} & {\af^{n-1}/\afnz{n-1}\wedge (\Gm \times V_{k-1}(\af^{n-1}))_+} & { } \\
	{(\afnz{n-1}\times V_{k-1}(\af^{n-1}))_+}
	\arrow[from=3-1, to=2-2]
	\arrow[from=2-2, to=2-3]
	\arrow["\alpha", from=2-3, to=2-4]
	\arrow["{j_2}", from=1-2, to=2-2]
	\arrow[Rightarrow, no head, from=1-3, to=2-3]
	\arrow[from=1-1, to=1-2]
	\arrow[from=1-2, to=1-3]
	\arrow[from=2-1, to=2-2]
	\arrow["\cong"{description}, from=2-1, to=3-1]
	\arrow["{j_1}", from=1-1, to=2-1]
	\arrow[ from=1-3, to=1-4]
	\arrow["\gamma"', curve={height=18pt}, from=1-1, to=3-1]
\end{tikzcd}\]}

In the above diagram, $\gamma:Q_{2n-3}\times \Gm\times V_{k-1}(\af^{n-1})\to Q_{2n-3}\times V_{k-1}(\af^{n-1}) $ is the morphism defined by 

	\[ (\mathbf{x},\mathbf{a},a_n,\mathbf{B})\mapsto
	(\mathbf{x},\mathbf{a},\mathbf{B}\mathbf{M}(\mathbf{x},\mathbf{a},a_n)^T)
		\]
where $\mathbf{M}(\mathbf{x},\mathbf{a},a_n)=\mathbf{I}+(a_n-1)\mathbf{a}^T\mathbf{x}$.

In analogy with the previous section, we expect to get split decompositions.
\begin{proposition}
	Using our inductive hypothesis and $\mrm{PH}(k,n)$, we conclude that we have the following split cofiber sequences in $\Da(K)[\eta^{-1}]$

 \[V_{k-1}(\af^{n-1})_+ \to {V_k(\af^n)}_+ \to \afnz{n} \wedge V_{k-1}(\af^{n-1})_+ \xr{0 } \]
for $n$ even, and
\[V_{k-2}(\af^{n-2})_+ \to {V_k(\af^n)}_+ \to \afnz{n} \wedge \afnz{n-1} \wedge  V_{k-2}(\af^{n-2})_+ \xr{0 }, \]
for $n$ odd. In particular, $\mrm{P}(k,n)$ holds.
\end{proposition}
\begin{proof}
  If $\mrm{PH}(k,n)$ is assumed to be true, it allows us to construct split surjections in $\eta$-inverted cohomology. These can be used, as in the proof of Theorem \ref{SpMain}, to split the cofiber sequences in the statement of the Proposition, once we have proved that the cofiber sequences actually exist. For $n$ even, there is nothing to do. For $n=2i+1$ odd, consider the following diagram:
	\[\begin{tikzcd}
		{V_{k-2}(\af^{n-2})_+} & {V_{k-1}(\af^{n-1})_+} & {\afnz{n-1}\wedge V_{k-2}(\af^{n-2})_+} \\
		{V_{k-2}(\af^{n-2})_+} & {{V_k(\af^n)}_+} & {C(\Sigma^{-1}\tau)} \\
		{V_{k-1}(\af^{n-1})_+} & {{V_k(\af^n)}_+} & {\afnz{n} \wedge V_{k-1}(\af^{n-1})_+ }\\
		\arrow[Rightarrow, no head, from=1-1, to=2-1]
		\arrow[from=2-1, to=3-1]
		\arrow["i", from=1-1, to=1-2]
		\arrow[dashed,"p",curve={height=-15pt},from=1-2, to=1-1]
		\arrow[from=2-1, to=2-2]
		\arrow[from=3-1, to=3-2]
		\arrow[from=1-2, to=1-3]
		\arrow[dashed,"j",curve={height=-15pt},from=1-3, to=1-2]
		\arrow[from=1-2, to=2-2]
		\arrow[Rightarrow, no head, from=2-2, to=3-2]
		\arrow[from=2-2, to=2-3]
		\arrow[from=3-2, to=3-3]
		\arrow[dashed, from=1-3, to=2-3]
		\arrow[dashed, from=2-3, to=3-3]
		\arrow["\ \ \tau"',"+1"{description}, curve={height=40pt},dotted, from=3-3, to=1-3]
	\end{tikzcd}\]

The morphism $\tau$ is simply the composite 
\[\tau:\afnz{n} \wedge V_{k-1}(\af^{n-1})_+ \xr{\beta} \Sigma V_{k-1}(\af^{n-1})_+\to  \Sigma\afnz{n-1}\wedge V_{k-2}(\af^{n-2})_+\]
and we now prove that it is split in $\Da(K)[\eta\inv]$, by constructing a morphism 
$$\mu:\afnz{n} \wedge V_{k-2}(\af^{n-1})_+\to \afnz{n} \wedge V_{k-1}(\af^{n-1})_+ $$ 
such that $\tau\mu$ is an isomorphism.
\begin{remark}
	Notice that $\mu$ shall NOT be $1\wedge i$. Instead, we will define a morphism analogous to $\Sigma d: \afnz{n} \to \afnz{n} \wedge \af^{n-1}_+$ in the previous section.
\end{remark}
Since $n-1$ is even, we have the split cofiber sequence 
\[
V_{k-2}(\af^{n-2})_+ \xr{i} {V_{k-1}(\af^{n-1})}_+ \to \afnz{n-1} \wedge V_{k-2}(\af^{n-2})_+
\]
for which we choose a splitting 
\[
j:\afnz{n-1} \wedge V_{k-2}(\af^{n-2})_+\to {V_{k-1}(\af^{n-1})}_+.
\]

Denoting by $f: V_{k-1}(\af^{n-1})\to \afnz{n-1} $ the first row map, we define a morphism of schemes 
\[
d': \Gm \times V_{k-1}(\af^{n-1})\to\Gm \times\afnz{n-1}\times V_{k-1}(\af^{n-1})
\]
by $d'(x,y)=(x,f(y),y)$, and we obtain a morphism in $\Da(K)[\eta\inv]$ of the form
\[ d: \Gm\wedge \afnz{n-1} \wedge V_{k-2}(\af^{n-2})_+ \xr{1\wedge j} \Gm \wedge V_{k-1}(\af^{n-1})_+\xr{d'}\Gm \wedge \afnz{n-1}\wedge V_{k-1}(\af^{n-1})_+.\]

Let $\mu=\Sigma d$, we can now prove that the composition
\[ \afnz{n} \wedge V_{k-2}(\af^{n-2})_+ \xr{\Sigma d} \afnz{n} \wedge V_{k-1}(\af^{n-1})_+  \xr{\tau} \Sigma \afnz{n-1}\wedge V_{k-2}(\af^{n-2})_+ \]
is simply $\eta \wedge 1 $.

\begin{lemma}
	The composition 
\[ \Gm \times V_{k-1}(\af^{n-1})\xr{d'}
\afnz{n-1}\times \Gm \times V_{k-1}(\af^{n-1})
\xr{pr_2\circ\gamma} V_{k-1}(\af^{n-1})
\]
is homotopic to the morphism 
\[ (c,\begin{bmatrix}
	a_1& \cdots & a_{n-1}\\
	\ddots \\
	b_{j1}& \cdots&b_{j,n-1}\\
	\ddots
	
\end{bmatrix})\mapsto
\begin{bmatrix}
	ca_1& \cdots & ca_{n-1}\\
	\ddots \\
	b_{j1}& \cdots&b_{j,n-1}\\
	\ddots
	
\end{bmatrix} \]
Thus this morphism induces 
\[\eta \wedge 1: \afnz{n} \wedge V_{k-2}(\af^{n-2})_+ \xr{\Sigma d} \afnz{n} \wedge V_{k-1}(\af^{n-1})_+  \xr{\tau} \Sigma \afnz{n-1}\wedge V_{k-2}(\af^{n-2})_+ \]

Moreover, Let $p: V_{k-1}(\af^{n-1})_+ \to V_{k-2}(\af^{n-2})_+$ be a splitting, the composition

\[\afnz{n} \wedge V_{k-2}(\af^{n-2})_+ \xr{\Sigma d} \afnz{n} \wedge V_{k-1}(\af^{n-1})_+  \xr{1\wedge p} \afnz{n}\wedge V_{k-2}(\af^{n-2})_+ \]
is the identity.

\end{lemma}
 \begin{proof}
	First we define $QV_{k}(\af^{n}):\{(\mathbf{B},\mathbf{x})\in V_{k}(\af^{n})\times\af^{n}, s.t. \mathbf{B}\mathbf{x}^T=[1,0,\ldots,0]^T \}$. This gives a $\af^{n-k}$ bundle over $V_{k}(\af^{n})$. 

	Now we replace $V_{k-1}(\af^{n-1})$ with $QV_{k-1}(\af^{n-1})$ and we find that $d'$ is equivalent to the morphism 
	\[ \Gm \times QV_{k-1}(\af^{n-1})\to Q_{2n-3} \times \Gm \times V_{k-1}(\af^{n-1}),(c,\mathbf{B},\mathbf{x})\mapsto ((\mathbf{x},f(\mathbf{B})),c,\mathbf{B})
	\]
	Following the definition of $\gamma$, we have 
	\[pr_2\circ\gamma((\mathbf{x},f(\mathbf{B})),c,\mathbf{B}) =B+(c-1)[1,0,\ldots,0]^Tf(\mathbf{B})\]
	which is exactly the matrix
	\[\begin{bmatrix}
		ca_1& \cdots & ca_{n-1}\\
		&\ddots \\
		b_{j1}& \cdots&b_{j,n-1}\\
		&&\ddots
		
	\end{bmatrix}\]

 Thus the morphism $ \tau \circ \Sigma d: \Gm \wedge \Sigma \afnz{n-1} \wedge V_{k-2}(\af^{n-2})_+ \to \Sigma \afnz{n-1}\wedge V_{k-2}(\af^{n-2})_+$
  is induced by 

  \[
    (c, \mathbf{a}, \begin{bmatrix}
		b_1& \cdots & b_{n-2}\\
		&\ddots \\
		b_{j1}& \cdots&b_{j,n-2}\\
		&&\ddots
		
	\end{bmatrix}) \mapsto
    (c,\mathbf{a}, \begin{bmatrix}
      a_1& \cdots & a_{n-2} & a_{n-1}\\
      b_1& \cdots & b_{n-2} & 0\\
		&\ddots \\
      b_{j1}& \cdots&b_{j,n-2} & 0\\
		&&\ddots
		
	\end{bmatrix}) \mapsto
    (c\mathbf{a}, \begin{bmatrix}
		b_1& \cdots & b_{n-2}\\
		&\ddots \\
		b_{j1}& \cdots&b_{j,n-2}\\
		&&\ddots
		
	\end{bmatrix}),
  \]
  where $ \mathbf{a}=(a_1,\cdots, a_{n-1}) $. We can conclude it is $ \eta \wedge 1 $ by applying \cite[Lemma 2.48]{morel2020cellular} again.

  Finally, the composition
\[ \Gm \wedge \Sigma \afnz{n-1} \wedge V_{k-2}(\af^{n-2})_+ \xr{\Sigma d} \Gm \wedge \Sigma \afnz{n-1} \wedge V_{k-1}(\af^{n-1})_+  \xr{1\wedge p} \Gm \wedge \Sigma \afnz{n-1}\wedge V_{k-2}(\af^{n-2})_+ \]
is induced from 
\[
    (c, \mathbf{a}, \begin{bmatrix}
		b_1& \cdots & b_{n-2}\\
		&\ddots \\
		b_{j1}& \cdots&b_{j,n-2}\\
		&&\ddots
		
	\end{bmatrix}) \mapsto
    (c,\mathbf{a}, \begin{bmatrix}
      a_1& \cdots & a_{n-2} & a_{n-1}\\
      b_1& \cdots & b_{n-2} & 0\\
		&\ddots \\
      b_{j1}& \cdots&b_{j,n-2} & 0\\
		&&\ddots
		
	\end{bmatrix}) \mapsto
  (c, \mathbf{a}, \begin{bmatrix}
		b_1& \cdots & b_{n-2}\\
		&\ddots \\
		b_{j1}& \cdots&b_{j,n-2}\\
		&&\ddots
		
	\end{bmatrix}).
  \]
 \end{proof}

 With the same argument as in Remark \ref{etaSplit} we have $ C(\Sigma d)\cong\afnz{n} \wedge \afnz{n-1} \wedge  V_{k-2}(\af^{n-2})_+ $. Finally, since $\tau \circ \Sigma d = \eta \wedge 1$ is an isomorphism, we can conclude that $ C(\Sigma^{-1}\tau)\cong C(\Sigma d)\cong\afnz{n} \wedge \afnz{n-1} \wedge  V_{k-2}(\af^{n-2})_+$.

\end{proof}

Now, it only remains to prove $\mrm{PH}(k,n)$. This will require some delicate arguments that we start developing in the next section.

\subsection{Recollections on \MW cycle module}
\label{MWmod}
We can construct a t-structure $\delta$ on $\SH(S)$ \cite{bondarko2017dimensional}, whose heart can be understood using either homotopy modules \cite{deglise2011modules} or equivalently (cohomological) \MW(MW)-cycle modules (\cite{feld2020milnor} or \cite{deglise2022perverse}). 

We denote by $\mathcal{F}_S$ the category of fields essentially of finite type over $S$. A (cohomological) \MW cycle module is a functor from $\mathcal{F}_S$ to the category of $\bZ$-graded abelian groups $\mathrm{Abel}^{\bZ}$
\[M_*: \mathcal{F}_S \to \mathrm{Abel}^{\bZ} \] 
satisfying the axioms listed in \cite[\S 5.1]{deglise2022perverse}. The axioms that $M_*$ should satisfy include an action by Milnor-Witt $K$-theory $\KMW_*$, which is the fundamental example of a MW-cycle module. In the sequel, we will denote by $\mathrm{MWmod}^{coh}(S)$ the category of cohomological MW-cycle modules over $S$. We note that the dual notion of homological MW-cycle module is developed in \cite{deglise2022perverse}, but we choose the former notion to fit our computations.

For any $p: X\to S$ and any \MW cycle module $M_* \in \mathrm{MWmod}^{coh}(S)$, there is an associated Gersten complex $^{\delta}C^*(X, M_*)$, which actually defines a spectrum $ ^{\delta}\mathbf{H}M\in \SH(S)$ satisfying the following property 
\[
	{}^{\delta}A^i(X,M_j):= H^i({}^{\delta}C^*(X,M_j))\cong {}^{\delta}\mathbf{H}M^{i+j,j}(X).
\] 
The Gersten complex gives the heart of the t-structure $\delta$. Correspondingly, we can also get a homotopy module, i.e. a $\bZ$-graded strictly $\af^1$-invariant sheaf $\mathbb{M}_*: U\mapsto{}^{\delta}C^0(U, M_*)$.

The above cohomology theory $\Heta$ gives an example, corresponding to the MW-cycle module $\mathbf{W}_*$ associated to the Witt sheaf. Here, $\mathbf{W}_i(k)$ is just the equivalence classes of non-degenerate quadratic forms up to hyperbolic planes. Sometimes, we omit the ${}^{\delta}\mathbf{H}$ while identifying the MW-cycle module and its corresponding spectrum.

We can also pull back MW-cycle modules along morphisms of schemes $p:X\to S$, obtaining an object $p^*(M)_* \in \mathrm{MWmod}^{coh}(X) $ given by
\[p^*(M)_*(  (v\to X) \in \mathcal{F}_X) := M_*(v\to X\xr{p}S).\]
We say that such a MW-cycle module $p^*(M)_*$ over $X$ is \emph{(cohomologically) $S$-simple}.

Next, we would like to define twists of a given MW-cycle module with respect to a principal bundle. If $G$ is an algebraic group (or more generally a sheaf of groups over $S$) and $M_* \in \mathrm{MWmod}^{coh}(S)$, we want to consider actions of $G$ on $M_*$ which preserve all the constraints satisfied by $M_*$. In the framework of this thesis, we will only consider the following notion. A \emph{compatible action of $G$ on $M_*$}  is a linear action $A(F): G(F)\times M_*(F)\to M_*(F)$ for any $F\in \mathcal{F}_S$, subject to the condition that for any discrete valuation ring $\mathcal{O}_v$ over $S$ (i.e. endowed with a morphism $\mathcal{O}(S)\to \mathcal{O}_v$) with fraction field $K(v)$ and residue field $\kappa(v)$
\[
K(v) \xleftarrow{\xi} \mathcal{O}_v \xr{s} \kappa(v)
\]
the following diagram commutes
\begin{equation}\label{eqn:action}
\begin{tikzcd}
	{G( \mathcal{O}_v)\times M_*(K(v))} & {G(K(v))\times M_*(K(v))} & {M_*(K(v))} \\
	& {G(\kappa(v))\times M_{*-1}(\kappa(v),N_v^\vee)} & {M_{*-1}(\kappa(v),N_v^\vee)}
	\arrow["{\xi_*\times \mathrm{id}}", from=1-1, to=1-2]
	\arrow["{s_*\times \partial}"', from=1-1, to=2-2]
	\arrow["{A(\kappa(v))}", from=2-2, to=2-3]
	\arrow["{A(K(v))}", from=1-2, to=1-3]
	\arrow["\partial", from=1-3, to=2-3]
\end{tikzcd}
\end{equation}
where $N_v^\vee$ is the dual of the normal bundle to $\kappa(v)$ in $\mathcal O_v$, $\delta$ is the residue map of Axiom (D4) in \cite[\S 5.1]{deglise2022perverse}, and the action $A(\kappa(v))$ of $G(\kappa(v))$ on $M_{*-1}(\kappa(v),N_v^\vee)$ is the natural one induced by the action on $M_{*-1}(\kappa(v))$ (the action on the line bundle is trivial).

Given a principal $G$-bundle $P$ over $X$, we set 
\[ (M\circledast _A P)_* ((x:v\to X) \in \mathcal{F}_X):=  P_x\otimes^A_{G_x}p^*(M)_*(x). \]
Notice that $P_x$ is non-canonically isomorphic to $G_x$, and thus $(M\circledast_A P)_* (x)$ is also non-canonically isomorphic to $p^*(M)_*(x)$. A direct verification shows that the axioms of a cohomological MW-cycle module are verified (use \eqref{eqn:action} above for (D4)) and we therefore obtain a well-defined object $(M\circledast _A P)_*$ in the category of cohomological MW-cycle modules.

In the particular example where $G=\mathbb{G}_m$, we can make $G$ act on $M_*$ via the obvious map $\mathbb{G}_m\to \KMW_0$. If $L$ is a line bundle over $X$, it is an easy exercise to check that $M\circledast _A L$ is the usual cycle module twisted by $L$, with Gersten complex
\[
^{\delta}C_*(X, M\circledast_{A} L )={}^{\delta}C_*(X, M, L).
\]
A compatible action of $G$ on $M_*$ also yields an action on the Gersten complex asscociated to $M_*$. Indeed, for any $p,q\in\mathbb{Z}$ and any dimension function $\delta$, we obtain a string of maps for any open subscheme $U\subset X$
\[
G(U)\times\bigoplus_{\delta(x)=p} M_{p+q}(\kappa(x),\lambda_x)\xr{\oplus s_{x}\times \mathrm{id}}\bigoplus_{\delta(x)=p} G(\kappa(x))\times M_{p+q}(\kappa(x),\lambda_x)\xr{\oplus A(\kappa(x))}\bigoplus_{\delta(x)=p} M_{p+q}(\kappa(x),\lambda_x).
\]
The fact that Diagram \eqref{eqn:action} commutes precisely shows that this action commutes with the differentials of the complex, and we thus obtain an induced action
\[
G(U)\times {}^{\delta}C_*(U,M_q)\to  {}^{\delta}C_*(U,M_q)
\] 
for any $q\in\mathbb{Z}$. In particular, $G$ acts on the spectrum $^{\delta}\mathbf{H}M$ associated to $M_*$.
\begin{remark}
	We could also define the twist of a homotopy module $\mathbb{M}$ in a similar way. 
\end{remark}

An $\eta$-inverted cycle module $M_{\eta}$ is a MW-cycle module for which $\eta \in \KMW_{-1}(K)$ acts as an isomorphism. The MW-cycle module $\mathbf{W}$ is such an example and one can check that it is the tensor unit in the category of $\eta$-inverted cycle modules. More generally, we can define the $\eta$-localization of any given MW-cycle module $M$ as
$$M[\eta^{-1}]_*:=\mathbf{W}_* \otimes_{\KMW_*} M_* .$$
This is obviously an $\eta$-inverted cycle module, and one can check that this functor is adjoint with the inclusion of the category of $\eta$-inverted MW-cycles modules into the category of MW-cycle modules. 

Finally, let us observe that given an $\eta$-inverted cycle module $M_*$, we only need to consider $(M_{\eta})_0$. From now on, we only consider $\eta$-inverted MW-cycle modules.

\subsection{Recollections on the motivic Leray spectral sequences}

To compute cohomology using fiber sequences, we can use the motivic analogue of the Leray spectral sequence by making use of the t-structure $\delta$ \cite{asok2018homotopy}.

Let $H^p_{\delta}$ be the cohomology of the t-structure $\delta$, which takes value in the heart, i.e. MW-cycle modules. In general, for any spectrum $\mathbb{F}\in \SH(S)$, a spectral sequence can be obtained using the theory of spectral diagrams: 
\[ E_{2}^{p,q}(\mathbb{F})={}^{\delta}A^p(S,H_{\delta}^q \mathbb{F})\Rightarrow H^{p+q}(S,\mathbb{F})\]

Applying to the push-forward $f_*\mathbb{E}$ along a morphism $f:T \to B$, we obtain:
\begin{theorem}[\cite{asok2018homotopy},  Theorem 4.2.5]
	Let $f: T \to B $ be a morphism of schemes, and let $\mathbb{E}\in \SH(T)$ be a ring spectrum over $T$, we have a spectral sequence of the form:
	\[ E_{2}^{p,q}(f,\mathbb{E})={}^{\delta}A^p(B,H_{\delta}^q f_*\mathbb{E})\Rightarrow H^{p+q}(T,\mathbb{E})\]
where the differentials $d_r$ satisfy the usual Leibniz rule.
\end{theorem}

When $\mathbb{E}=f^*\mathbb{F}$ for some $ \mathbb{F}\in \SH(B) $, we have the non-canonical isomorphism
\[
(H^p_{\delta}f_*\mathbb{E})_q(b:v\to B)\cong H^{p+q}(T_b, b^!\mathbb{F}),
\] 
where $T_b := v\times_B T $.

However, computing $H^p_{\delta}f_*\mathbb{E}$ as a MW-cycle module is challenging in general. In our scenario, we focus on the locally trivial fibration 
\[
F\to T\xrightarrow{f} B
\] 
over $S$ using an $\eta$-inverted cycle module $M\in \mathrm{MWmod}^{coh}(S)$ as our coefficients. 

In classical computations, simple connectedness is used to further reduce $\mathbf{R}^q f_*\mathbb{Z}$ into the constant sheaf associated to $\mathrm{H}^q(F, \mathbb{Z})$ over $B$. However, additional care is necessary in the motivic setting. Instead, we show that $H_{\delta}^q f_*\mathbb{E}$ is (cohomologically) $S$-simple, i.e. $H_{\delta}^q f_*\mathbb{E}\cong p_{B}^*(N)$, for $N\in \mathrm{MWmod}^{coh}(S)$

\begin{lemma}
  \label{pushFwdSimple}
  Let $f:T\to B$ be a $\tau$-locally trivial fibration($\tau$ can be $Zar, Nis, et$), with the fiber $F$ such that $\mathrm{H}^*(F, \mathbf{W})$ is a flat $\mathrm{H}^*(S, \mathbf{W})$-module.
  If the non-abelian cohomology class $ c_f \in \mathrm{H}_{\tau}(B, \mathrm{Aut}(\mathrm{H}^q(F, \mathbf{W})\otimes_\mathbf{W} M)) $ induced by $f$ is trivial, then $H^q_{\delta}f_*M$ is simple.
\end{lemma}
\begin{proof}
   
We begin by considering a $\tau$-covering $U \to B$ satisfying 
\[
U \times_B T \xrightarrow[\cong]{\chi} U \times F,
\] 
where $\chi$ is a $U$-morphism. The assumption that $\mathrm{H}^*(F, \mathbf{W})$ is a flat $\mathrm{H}^*(S, \mathbf{W})$-module gives  
\[
\mathrm{H}^*(V \times_S F, M) = \mathrm{H}^*(V, M) \otimes_{\mathrm{H}^*(S, \mathbf{W})} \mathrm{H}^*(F, \mathbf{W})
\] 
for any $V$ over $S$. In order to comprehend the MW-cycle module $H^p_{\delta}f_*M$, we use the trivialization $\chi: U\times_B T \to U\times F$ and the trivial $F$-bundle $f_U: U\times_B T \to U$. When we compare the trivializations induced on the two projections, namely $\mathrm{pr}_1,\mathrm{pr}_2:U\times_B U\to U$, we obtain the transformation morphism 
\[
g: U\times_B U\times F\xrightarrow{\mathrm{pr}_1^* \chi^{-1}} U\times_B U\times_B T\xrightarrow{\mathrm{pr}_2^* \chi}U\times_B U\times F.
\] 

We observe that the $\eta$-inverted cycle module $H^p_{\delta}f_{U*}M$ over $U$  is simple, as $\mathrm{H}^q(F_{(-)},M)\cong \mathrm{H}^q(F, \mathbf{W})\otimes_\mathbf{W} M $. Therefore, the $\eta$-inverted cycle module $ H^p_{\delta}f_*M $ over $B$ is locally simple. It follows from our assumption that, for a local (or more generally semi-local) $V\subset U$, $\mathrm{H}^q(V\times F,M)= \mathrm{H}^0(V,M)\otimes \mathrm{H}^q(F,\mathbf{W})$. Actually $H^p_{\delta}f_M$ is the twist $ (\mathrm{H}^q(F, \mathbf{W})\otimes_\mathbf{W} M)\circledast_{g} P_T $ for some  $\mathrm{Aut}(\mathrm{H}^q(F, \mathbf{W})\otimes_\mathbf{W} M)$ principal bundle $P_T$ over $B$. Moreover, $g$ induces an automorphism 
\[
g^*: \mathrm{H}^*(U\times_B U,M)\otimes \mathrm{H}^*(F,\mathbf{W})\to \mathrm{H}^*(U\times_B U,M)\otimes \mathrm{H}^*(F,\mathbf{W})
\] 
that provides the $\mathrm{Aut}(\mathrm{H}^q(F, \mathbf{W})\otimes_\mathbf{W} M)$ principal bundle $P_T$ gluing the cycle module $H^p_{\delta}f_*M$. And by definition the class $c_f:=[g^*]$.

Finally, if the class $c_f$ is trivial, then there is another automorphism, $h:\mathrm{H}^q(U\times F,M)\to \mathrm{H}^q(U\times F,M)$ such that $g^*=\mathrm{pr}_1^*h^{-1}\circ\mathrm{pr}_2^*h$, we may conclude that $P_T$ is indeed a trivial $\mathrm{Aut}(\mathrm{H}^q(F, \mathbf{W})\otimes_\mathbf{W} M)$ principle bundle, and thus since $\mathrm{H}^q(F, \mathbf{W})\otimes_\mathbf{W} M$ is simple, so is $H^q_{\delta}f_*M$.

\end{proof}
\begin{remark}
	Although Higher-order coherence data exists, in this instance, the $\mathrm{Aut}(\mathrm{H}^q(F, \mathbf{W})\otimes_\mathbf{W} M)$-principal bundle is the only variable necessary for consideration due to the cohomology groups landing in $\mathrm{Set}$. Later in our scenario, we will even reduce to the case of a $\Gm$-principal bundle, i.e. a line bundle.
\end{remark}

\subsection{Leray spectral sequence and Stiefel varieties}

In the above discussion, we only need to consider the case $M=\mathbf{W}$, so we may omit the coefficient $ \mathbf{W} $ for simplicity.

\subsection{First case of Stiefel varieties}
We begin with the special cases of the fiber sequences (\ref{fibSV}) for Stiefel varieties 
\[V_{k-1}(\af^{n-1})\to V_k(\af^n)\xr{f} \afnz{n}\]
\[f: \begin{bmatrix}
	a_1& \cdots & a_n\\
	&\ddots \\
	b_{j1}& \cdots&b_{jn}\\
	&&\ddots
	
\end{bmatrix}\mapsto [a_1,\cdots, a_n]=\mathbf{a}.\]
We follow the processes above for these fiber sequences. Note that 
\[
\mathrm{H}^*(S\times V_{k-1}(\af^{n-1}),\mathbf{W})=\mathrm{H}^*(S,\mathbf{W})\otimes \mathrm{H}^*(V_{k-1}(\af^{n-1}),\mathbf{W})
\] 
is already given by $\mathrm{PH}(k-1,n-1)$

First we consider the covering $U=\sqcup U_i \to \afnz{n}$ where $U_i=\{a_i\neq 0\}\cong \af^{i-1}\times \Gm \times \af^{n-i}$. The trivializations $\chi_i : U_i\times_{\afnz{n}} V_k(\af^n)\to U_i\times V_{k-1}(\af^{n-1})$ are given by the following matrix representations
\begin{eqnarray*}
	\begin{bmatrix}
		a_1\circ& \cdots& a_i* & \cdots  & a_n\circ\\
		&\ddots \\
		b_{j1}& \cdots&b_{ji}& \cdots&b_{jn}\\
		&&&\ddots
		
	\end{bmatrix}
	&\mapsto& \begin{bmatrix}
		a_1\circ& \cdots& a_i* & \cdots  & a_n\circ\\
		&\ddots \\
		b_{j1}-a_1 b_{ji}/a_i& \cdots& 0& \cdots&b_{jn}-a_n b_{ji}/a_i\\
		&&&\ddots
	\end{bmatrix}\\
	&=&([a_1\circ, \cdots, a_i* , \cdots,a_n\circ],\begin{bmatrix}
		
		\ddots \\
		b_{j1}-a_1 b_{ji}/a_i& \cdots& b_{jn}-a_n b_{ji}/a_i\\
		&\ddots
	\end{bmatrix}).
\end{eqnarray*}
Let $C(i)$ be the matrix $\begin{bmatrix}
		
	\ddots \\
	b_{j1}-a_1 b_{ji}/a_i& \cdots& b_{jn}-a_n b_{ji}/a_i\\
	&\ddots
\end{bmatrix}$.
We aim to understand the variation of $C(i)$ between two open subschemes, that is, between $U_i$ and $U_k$. We use the notation $U_{ik}=U_i\cap U_k$ with $i<k$. We compare the formulas for $C(i)$ and $C(k)$ to obtain the relation $C(i)G_{ik}(\mathbf{a})E_{ik}=C(k)$ where $E_{ik}$ is the permutation matrix of columns given by $(i\ i+1\ \cdots\ k-1)$ and
\[ G_{ik}(\mathbf{a})=\begin{bmatrix}
	1&&&&&&&&0 \\
	&\ddots\\
	&& 1\\
	&&&1\\
	&&&&\ddots\\
	0&&&&&1\\
	-a_1/a_k &\cdots &-a_{i-1}/a_k&-a_{i+1}/a_k& \cdots&-a_{k-1}/a_k& -a_i/a_k &\cdots & -a_n/a_k\\
	&&&&&&&\ddots
\end{bmatrix}.\]
Here, $G_{ik}(\mathbf{a})_{k-1,k-1}=-a_i/a_k$; for $j<i$, $G_{ik}(\mathbf{a})_{k-1,j}=-a_j/a_k$; for other $j$, $G_{ik}(\mathbf{a})_{k-1,j}=-a_{j+1}/a_k$, and otherwise $G_{ik}(\mathbf{a})_{pq}=\delta_{pq}$. From this, we obtain the transformation morphisms $g_{ik}: U_{ik}\times V_{k-1}(\af^{n-1})\to U_{ik}\times V_{k-1}(\af^{n-1})$

Up to $\af_1$-homotopy, $G_{ik}(\mathbf{a})\sim \mrm{diag}(-a_i/a_k,1,\ldots,1)$, and $\det(E_{ik})=(-1)^{k-i-1}$, hence $G_{ik}(\mathbf{a})E_{ik}\sim D((-1)^{k-i}a_i/a_k):=\mrm{diag}((-1)^{k-i}a_i/a_k,1,\ldots,1)$. 

\begin{remark}
	We observe that only the columns $i$ and $k$ are relevant, hence we only need to consider the transformation restricted to those columns. This crucial observation will be helpful in the sequel, when we will have to understand the general case.
\end{remark}

As we mentioned above, the transformation morphisms induce morphisms of cohomology groups, which are defined up to homotopy. Therefore, we obtain 
\[
g^*_{ik}=D((-1)^{k-i}a_i/a_k)^*:\mathrm{H}^q(U_{ik}\times V_{k-1}(\af^{n-1}))\to \mathrm{H}^q(U_{ik}\times V_{k-1}(\af^{n-1})).
\] 
To show that the torsor is trivial, define $h_i=D((-1)^{i}a_i^{-1}): U_{i}\times V_{k-1}(\af^{n-1})\to U_{i}\times V_{k-1}(\af^{n-1})$, and obviously, $g^*_{ik}=h_i^{-1*}h_k^*$.

In conclusion, applying lemma \ref{pushFwdSimple}, we have shown that $H_{\delta}^q f_*\mathbf{W} $ is the simple cycle module $\mathrm{H}^q(V_{k-1}(\af^{n-1}))\otimes \mathbf{W}$ on $\afnz{n}$.


\subsection{More general cases}
\label{trivialMono}

We can perform similar computations for all fiber sequences of the form

\[V_{k-l}(\af^{n-l})\to V_k(\af^n)\xr{f_l} V_l(\af^n),\]
\[f_l: \begin{bmatrix}
	\mathbf{a}_{L1}& \cdots & \mathbf{a}_{Ln}\\
	&\ddots \\
	b_{j1}& \cdots&b_{jn}\\
	&&\ddots
	\end{bmatrix}:=\begin{bmatrix}
	a_{11}& \cdots & a_{1n}\\
	&\ddots \\
	a_{l1}& \cdots & a_{ln}\\
	&\ddots \\
	b_{j1}& \cdots&b_{jn}\\
	&&\ddots
	\end{bmatrix}\mapsto
	\begin{bmatrix}
		\mathbf{a}_{L1}& \cdots & \mathbf{a}_{Ln}
	\end{bmatrix}
	=
\begin{bmatrix}
	a_{11}& \cdots & a_{1n}\\
	&\ddots \\
	a_{l1}& \cdots & a_{ln}\\
\end{bmatrix}.\]
We write respectively $\mathbf{a}_{Si}$ and $\mathbf{b}_{jS}$ for the column vector
$\begin{bmatrix}\vdots \\ a_{ti}\\ \vdots\end{bmatrix}$ and row vector 
$\begin{bmatrix}\cdots & b_{ti}& \cdots\end{bmatrix}$, where $t\in S$ denotes a subset of index.

Let $I$ be an $l$-subset of $\{1,\ldots,n\}$. Given a $l\times l$ matrix 
\[ \mathbf{A}_I =\begin{bmatrix}
	\cdots&\mathbf{a}_{Lt}& \cdots \\
	
\end{bmatrix}=\begin{bmatrix}
	\cdots&a_{1t}& \cdots \\
	&\ddots \\
	\cdots&a_{lt}& \cdots\\
\end{bmatrix} ,t\in I,
	\]
we have the open cover of $  V_l(\af^n) $ by $U=\sqcup U_I$ where $U_I=\{\det(\mathbf{A}_I)\neq0 \}\cong \GL_l$. We could refine the cover $U$ further to make it cohomologically trivial, but this does not affect the computations here. 

Now, let us consider the trivializations $\chi_I: f_l\inv(U_I)\to U_I\times  V_{k-l}(\af^{n-l})$, which are given by the following expression:
	\begin{eqnarray*}
		\begin{bmatrix}
			\mathbf{a}_{L1}\circ& \cdots & \mathbf{A}_I* & \cdots& \mathbf{a}_{Ln}\circ\\
			&\ddots \\
			b_{j1}& \cdots& \mathbf{b}_{jI}& \cdots&b_{jn}\\
			&&&\ddots
			\end{bmatrix}
			&\mapsto&
			\begin{bmatrix}
				\mathbf{a}_{L1}\circ& \cdots & \mathbf{A}_I* & \cdots& \mathbf{a}_{Ln}\circ\\
				&\ddots \\
				b_{j1}- \mathbf{b}_{jI}\mathbf{A}_I\inv \cdot \mathbf{a}_{L1}& \cdots& 0 & \cdots&b_{jn}- \mathbf{b}_{jI}\mathbf{A}_I\inv \cdot \mathbf{a}_{Ln}\\
				&&&\ddots
				\end{bmatrix}
				\\
			&=&(\begin{bmatrix}
				\mathbf{a}_{L1}\circ& \cdots & \mathbf{A}_I* & \cdots& \mathbf{a}_{Ln}\circ
				\end{bmatrix}, 
				\begin{bmatrix}
					
					\ddots \\
					\cdots&b_{js}- \mathbf{b}_{jI}\mathbf{A}_I\inv \cdot \mathbf{a}_{Ls} & \cdots\\
					&&\ddots
					\end{bmatrix})
				, s\notin I 
	\end{eqnarray*}

To compare two trivializations (up to homotopy) over two open subschemes $U_I$ and $U_K$, we only need to focus on the columns in the union $I\cup K$. Given the symmetry of the rows, we pick a row $j$ to construct the transformation from $\mathbf{c}_{jK}$ to $\mathbf{c}_{jI}$. Here $\mathbf{c}_{jK}=[\ldots\ b_{jk}- \mathbf{b}_{jI}\mathbf{A}_I\inv \mathbf{a}_{Lk} \  \ldots]_{k\in K, k\notin I} =P_{K\backslash I}(\mathbf{b}_{jK}-\mathbf{b}_{jI}\mathbf{A}_I\inv \mathbf{A}_{K})$ and $\mathbf{c}_{jI}=P_{I \backslash K}(\mathbf{b}_{jI}-\mathbf{b}_{jK}\mathbf{A}_K\inv \mathbf{A}_{I})$, where $P_S$ is the first projection to the subspace spanned by $S$ before embedding it into the entire row vector space $\af^{n-l}$.

Hence, the transformation from $\mathbf{c}_{jK}$ to $\mathbf{c}_{jI}$ is given by $ _{K\backslash I}|-\mathbf{A}_K\inv \mathbf{A}_{I}|_{I\backslash K}E_{IK}$, where $_S|\mathbf{A}|_T$ denotes the operation of selecting $S$-rows and $T$-columns of $\mathbf{A}$ before embedding it into $\GL_{n-l}$. Finally, $E_{IK}$ is the permutation that maps $K$-columns to $I$-columns.
 
Although we restrict the matrix to a portion, one simple observation is that $\mathbf{A}_I$ and $\mathbf{A}_K$ share some columns. By using a permutation $\epsilon$, we have 
\[
\mathbf{A}_{I}-\epsilon\mathbf{A}_K= \begin{bmatrix}
	\cdots&\mathbf{a}_{Lt}& \cdots & 0 & \cdots
	\end{bmatrix}
\] 
and therefore
\[
\mathbf{A}_K\inv\mathbf{A}_{I}=\epsilon+ \begin{bmatrix}
		\cdots& * & \cdots & 0 & \cdots
		\end{bmatrix}.
\] 
Consequently, we obtain 
		$ \det( _{K\backslash I}|\mathbf{A}_K\inv \mathbf{A}_{I}|_{I\backslash K})=\det(\mathbf{A}_K\inv \mathbf{A}_I) $.

Therefore the transformation $\mathbf{g}_{IK}$ can be expressed as $G_{IK}(\mathbf{A})E_{IK}\sim \mrm{diag}(\det(-\mathbf{A}_K\inv \mathbf{A}_I),1,\ldots,1)E_{IK}$. After some calculations, we obtain 
\[
\det(E_{IK})=\prod_{k_1<\ldots<k_l \in K,i_1<\ldots<i_l \in I}\det(E_{i_t k_t}) =(-1)^{\sum_{k_1<\ldots<k_l \in K,i_1<\ldots<i_l \in I} k_t-i_t-1}=(-1)^{(\sum_{k\in K} k -\sum_{i\in I} i) -l},
\] 
and finally, the transformation matrix up to homotopy is nothing but 
\[\mathbf{g}_{IK}\sim G_{IK}(\mathbf{A})E_{IK}\sim \mathbf{D}((-1)^{\sum_{k\in K} k -\sum_{i\in I} i}\det(\mathbf{A}_K)\inv \det(\mathbf{A}_I)).\]

Using the same method as above, we can define $\mathbf{h}_I= \mathbf{D}((-1)^{\sum_{i\in I} i}\det(\mathbf{A}_I)\inv)$ and $\mathbf{g}_{IK}^*=\mathbf{h}_I^{-1*}\mathbf{h}^*_K $ and we obtain the desired outcome.

Finally, we have shown as in the previous section that $H_{\delta}^q f_{l*}\mathbf{W} $ is the simple cycle module $\mathrm{H}^q(V_{k-1}(\af^{n-1}))\otimes \mathbf{W}$ on $V_{l}(\af^n)$.

\begin{remark}
	In Section \ref{SpSV}, we will demonstrate a conormal bundle that exhibits $\SL$-orientation, as supported also by this result.
\end{remark}

\subsection{Convergence of the relevant spectral sequences}
We now examine the following special cases of the fiber sequences (\ref{fibSV}):
\[ 
\afnz{n-k+1}\to V_k(\af^n)\to V_{k-1}(\af^n) 
\]
for $n-k$ even, and
\[ 
V_2(\af^{n-k+2})\to V_k(\af^n)\to V_{k-2}(\af^n) 
\]
for $n-k$ odd.

For $\Heta={}^{\delta}\mathbf{HW}$ we obtain when is $n-k$ even, 
{\small\[
E^{p,q}_2= \mathrm{H}^p(V_{k-1}(\af^{n}),\mathrm{H}^q(\afnz{n-k+1},\mathbf{W}))\cong \mathrm{H}^p(V_{k-1}(\af^{n}),\mathbf{W})\otimes_{\mathrm{H}^*(S,\mathbf{W})}\mathrm{H}^q(\afnz{n-k+1}),\mathbf{W})\Rightarrow \mathrm{H}^{p+q}(V_k(\af^n),\mathbf{W}). 
\]}
On the other hand, when $n-k$ is odd,
{\small\[E^{p,q}_2= \mathrm{H}^p(V_{k-2}(\af^{n}),\mathrm{H}^q(V_{2}(\af^{n-k+2}),\mathbf{W}))\cong \mathrm{H}^p(V_{k-2}(\af^n),\mathbf{W})\otimes_{\mathrm{H}^*(S,\mathbf{W})}\mathrm{H}^q(V_{2}(\af^{n-k+2}),\mathbf{W})\Rightarrow \mathrm{H}^{p+q}(V_k(\af^n),\mathbf{W}). \]}

It is important to note that the spectral sequences in the two cases collapse at the $E_2$-page.
\begin{proposition}
	\label{SSconverge}
If we assume that the conditions $\mrm{PH}(k-1,n)$ and $\mrm{PH}(k-2,n)$ hold, then $\mrm{PH}(k,n)$ also holds.
\end{proposition}
\begin{proof}
	While we showed this for the case $k=2$, we assume $k>2$ here. 

	With the statement of $\mrm{PH}(k,n)$, we know that $\Heta^*$ forms a free algebra. Thus, using the Kunneth formula, we only need to demonstrate when $S=\Spec K$. In order to prove that these spectral sequences collapse at the $E_2$-page, it is enough to examine the generators, and show that all differentials degenerate. For this, recall that the differentials in the spectral sequence satisfy Leibnitz rule. If $n-k$ is even, we obtain

\[\small
\begin{tikzcd}[sep=tiny]
	{n-k} & {\gamma_{n-k}} && 0 && {\alpha_{n-1}\gamma_{n-k}} && {\beta_{4j-1}\gamma_{n-k}} \\
	\vdots \\
	0 & 1 && 0 && {\alpha_{n-1}} && {\beta_{4j-1}} \\
	& 0 & \cdots & {n-k+1} & \cdots & {n-1 (for\ n\ even)} &\cdots& {4j-1(n-k+1<2j<n)} & \cdots
	\arrow["{d_2}", from=1-2, to=3-4]
\end{tikzcd}\]
while for $n-k$ odd, we have
\[ \small
	\begin{tikzcd}[sep=tiny]
	{2(n-k)+1} & {\beta_{2(n-k)+1}} && 0 && {\alpha_{n-1}\beta_{2(n-k)+1}} && {\beta_{4j-1}\beta_{2(n-k)+1}} \\
	\vdots \\
	0 & 1 && 0 && {\alpha_{n-1}} && {\beta_{4j-1}} \\
	& 0 & \cdots & {2(n-k)+2} & \cdots & {n-1 (for\ n\ even)} & \cdots & {4j-1(n-k+2<2j<n)} & \cdots
	\arrow["{d_2}", from=1-2, to=3-4]
\end{tikzcd}\]
The claim follows.
\end{proof}

\section{Multiplicative Structures of Cohomology Rings}
\label{ring}

In this section, we will study the multiplicative structures and complete the proof of the computations of the cohomology rings $\HH^{*,*}$ and $\Heta^{*}$. As mentioned earlier, we only need to focus on the cases where the base is over the field $K$. For our purposes, let us introduce another index of degree: $(p,\{q\}):=(p+q,q)$.

\begin{remark}
\label{HHvainish}	
We recall from \cite[\S 1.3]{bachmann2020milnor} that for $\HH$, we have:
\[ \HH^{p,\{q\}}(K)=\begin{cases}
	0 &\text{if $p>0$.}\\
	\KMW_q(K) &\text{if $p=0$.} \\
	0 & \text{if $p<0$ and $q\leq 0$.}\\
\end{cases} \]

Furthermore, for $\Heta$ we obtain:
\[ \Heta^{p,\{q\}}(K)=\Heta^{p}(K)=\begin{cases}
	0 &\text{if $p\neq 0$}\\
	\mathrm{W}(K) &\text{if $p=0$} \\
\end{cases} \]
\end{remark}

More generally, we say that a cohomology theory $E$ is \textbf{coconnected} if $E^{p,\{q\}}(K)=0 $ for $p>0$.

\subsection{Basic examples}

Let us now consider the basic case of $\afnz{n}$ for some $n\geq 1$. 

\begin{lemma}
	\label{ringAn}
	Following the same notation as in Example \ref{caseAn}, for the coconnected cohomology theory $E$
	we have a ring isomorphism 
	$$E^{*,*}(\Gm)\cong E^{*,*}(K)[\theta_1]/(\theta_1^2-[-1]\theta_1) $$
	and for $n>1$
	$$E^{*,*}(\afnz{n})\cong E^{*,*}(K)[\theta_n]/(\theta_n^2). $$
\end{lemma}
\begin{proof}
	From Example \ref{caseAn}, we already know that this is an isomorphism of modules. 
	For $\Gm$, this is well-known \cite{morel2012a1}.
	For $n>1$, we only need to show $\theta_n^2=0$. Let $\theta_n^2= a + b \theta_n$, where $a\in E^{2n-2,\{2n\}}(K)$ and $ b\in E^{n-1,\{n\}}(K)$. By the definition of coconnectedness, we have $a=0$ and $b=0$.
\end{proof}

Thus this result can be applied to both $\HH$ and $\Heta$. 

With the same argument, we can show that $\Heta^*(V_2(\af^{2k+1}))\cong \Heta^*(K)[\beta_{4k-1}]/(\beta_{4k-1}^2)$, where $\beta_{4k-1} \in \Heta^{4k-1}$.

\subsection{Multiplicative Structures of $\Heta^*(V_k(\af^n))$}
We now aim to show that the multiplicative structure of $\Heta^*(V_k(\af^n))$ is as simple as can be. As mentioned earlier, over $K$, the coefficient ring $\Heta^*(K)$ is concentrated in degree $0$ with $\Heta^0(K)=\mathrm{W}(K)$.

\begin{proof}[Proof of the multiplicative structure in Theorem \ref{SVCohMain}]
The result is clear for all $n$ when $k=1$ and for $n$ odd when $k=2$. We follow the notations in the proof of Proposition \ref{SSconverge}. We just need to show inductively that the new generators satisfy $\beta_{2(n-k)+1}^2=0$ or $\gamma_{n-k}^2=0$ (except $\gamma_0^2=[-1]\gamma_0$). The degrees of $\beta_{2(n-k)+1}^2$ and $\gamma_{n-k}^2$ are too small to be products of two or more other generators. In other words, $\beta_{2(n-k)+1}^2$ and $\gamma_{n-k}^2$ must be of the form $ c\alpha_{n-1} \text{(for $n$ even)} +\sum c_j\beta_{4j-1}$. These must be equal to $0$ (or $\gamma_0^2=[-1]\gamma_0$) since the degree of the former is even while the degree of the latter is odd and we apply Remark \ref{HHvainish}.
\end{proof}

\subsection{Symplectic Stiefel Varieties}
\label{SpSV}

Following the same idea, we can define the symplectic Stiefel varieties $\Sp V_{2k}(\af^{2n}) \subset V_{2k}(\af^{2n})$ as the varieties of symplectic morphisms $\rHom_{Sym}(\af^{2k},\af^{2n}) $, where we give $\af^{2n}$ the canonical symplectic structure $\omega$. Similar to the case of Stiefel varieties, we can also identify symplectic Stiefel varieties as homogeneous spaces:
\[\Sp V_{2k}(\af^{2n}) \cong \Sp_{2n}/\Sp_{2(n-k)}.\]

We will show that it is enough to "add one row" to get $\Sp V_{2k+2}(\af^{2n})$ out of $\Sp V_{2k}(\af^{2n})$.

\begin{lemma}
	Let $f:V_{2k+1}(\af^{2n})\to V_{2k}(\af^{2n})$ be defined by taking the first $2k$ rows, and let $Y:= f^{-1}(\Sp V_{2k}(\af^{2n}))$. Let $g:\Sp V_{2k+2}(\af^{2n})\to Y$ be defined by taking first $2k+1$ rows, and let $X := g(\Sp V_{2k+2}(\af^{2n})) $. Then $X$ is a retract of $Y$, and $g: \Sp V_{2k+2}(\af^{2n}) \to X$ is an affine bundle over $X$. In particular, we have the homotopy equivalence $\Sp V_{2k+2}(\af^{2n})\cong Y$. 
\end{lemma}
\begin{proof}
  First, we notice that $X \subset Y$ consists of matrices that the $ 2k+1 $-th row is symplectic orthogonal to the other rows. For any matrix $ \mathbf{A}\in Y $, we can find a composition of elementary transformations $T(\mathbf{A})$ such that $ \mathbf{A}T(\mathbf{A})\in X $. By applying lemma \ref{Mretract}, this gives a retraction, thus $Y\cong X$.

  Then we can verify the claim locally. Let $ 1\leq j \leq n $ and $U_j\subset X$ be the open subvariety such that if $\mathbf{A}\in U_j$, let $\mathbf{A}_{2k+1}=\mathbf{x}$ be the $(2k+1)$-st row of $\mathbf{A}$, then the $j$-th entry of $\mathbf{x}$ is non-zero, $x_j\neq 0$.
	
  For the matrix $\mathbf{B}$ in the fiber $g^{-1}(\mathbf{A})$, the $(2k+2)$-nd row of $\mathbf{B}$ is of the form $\mathbf{y}+c_{2k+1}\mathbf{x}+ \mathbf{z}$, where $y_{n+j}=x_j^{-1}$, otherwise $y_i=0$, and $\mathbf{z}\bot_{\omega}\{\mathbf{x},\mathbf{y}\}$. This shows that the fiber should also be an affine space.
\end{proof}

Notice that when we take $k=0$ this shows that $\Sp V_2(\af^{2n})\cong \afnz{2n}$.

In the meantime, $Y$ can also be understood as the complement of a subbundle in a vector bundle over $\Sp V_{2k}(\af^{2n})$: Let $\Sp V_{2k}(\af^{2n})\times \af^{2n}$ denote the free addition of the $(2k+1)$-st row, and define the embedding
$$i: \Sp V_{2k}(\af^{2n})\times \af^{2k} \to \Sp V_{2k}(\af^{2n})\times \af^{2n}$$
 by simply

\[(\mathbf{A} , \mathbf{x} )\mapsto
\begin{bmatrix}
	\mathbf{A}\\
	\mathbf{A}\mathbf{x}
\end{bmatrix}.
\]

Clearly, $Y$ is isomorphic to the complement of the image of $i$. Moreover, we have the following localization cofiber sequence:

$$ \Sp V_{2k+2}(\af^{2n})\cong Y \to \Sp V_{2k}(\af^{2n}) \to  \mathrm{Th}(N_i) $$

where $N_i$ is the normal bundle of $i$. We claim that $N_i$ is $\SL$-oriented: in fact, $\mathrm{det}(N_i)$ is exactly the torsor we computed in Section \ref{trivialMono}, which is trivial. Since $\HH$ is $\SL$-oriented, in the computation of $\HH $ we can choose an orientation 
$$\tau \in \HH^{2n-2k,\{2n-2k\}}(\Sigma\afnz{2n-2k})\hookrightarrow \HH^{2n-2k,\{2n-2k\}}( \mathrm{Th}(N_i)), $$

and we can identify  $\HH^{p,\{q\}}( \mathrm{Th}(N_i)) $ with $ \HH^{p-2(n-k),\{q-2(n-k)\}}(\Sp V_{2k}(\af^{2n}))\tau$.

We now follow the same idea of \cite[Proposition 5]{williams2012motivic} to show the following theorem:

\begin{theorem}
	\label{VSpHmw}

The graded ring $\HH^{*,*}(\Sp V_{2k}(\af^{2n}))$ is a free $\HH^{*,*}(S)$-square-zero algebra generated by the classes $ \beta_j\in \HH^{2j-1,\{2j\}}(\Sp V_{2k}(\af^{2n})) $ for $n-k < \ldots \leq n$. In other words, we have an isomorphism of rings:

\[  \HH^{*,*}(\Sp V_{2k}(\af^{2n})) \cong \HH^{*,*}(K)[\beta_{n-k+1},\ldots,\beta_n]/(\beta_{n-k+1}^2,\ldots,\beta_n^2).\]

\end{theorem}
\begin{proof}
We will prove this result inductively. For the base case $k=1$, the result is already known.

Using the localization cofiber sequence, we obtain the long exact sequence

\[ \to \HH^{p-2(n-k),\{q-2(n-k)\}}(\Sp V_{2k}(\af^{2n}))\tau \xr{i_*} \HH^{p,\{q\}}(\Sp V_{2k}(\af^{2n})) \to \HH^{p,\{q\}}(\Sp V_{2k+2}(\af^{2n}))\to \]

First, we will show that $i_*=0$, which is equivalent to showing that $ i_*(\tau)=0 \in \HH^{2(n-k),\{2(n-k)\}}(\Sp V_{2k}(\af^{2n}))$. Instead, we will demonstrate that $ \HH^{2(n-k),\{2(n-k)\}}(\Sp V_{2k}(\af^{2n})) = 0 $.

By our inductive hypothesis, all the generators of $\HH^{*,*}(\Sp V_{2k}(\af^{2n}))$ have degree at least $ (2(n-k)+1, \{2(n-k)+2 \}) $. Therefore, for any $c \in  \HH^{2(n-k),\{2(n-k)\}}(\Sp V_{2k}(\af^{2n}))$, we can write $c =\sum c_I \beta_I$. Moreover, by Remark \ref{HHvainish}, all the coefficients $c_I \in \HH^{p,\{q\}}(K)=0$ since $p,q<0$.

\begin{remark}
	
	This argument cannot be applied to $\HH$ of Stiefel varieties, as in that case, it is possible for $c$ to be of the form $c= c_0 \gamma_{n-k}$, where $c_0\neq 0 \in \HH^{0,\{-1\}}(K)$.

\end{remark}

This gives us the following module isomorphism:

\[ \HH^{*,*}(\Sp V_{2k+2}(\af^{2n})) \cong  \HH^{*,*}(\Sp V_{2k}(\af^{2n})) \oplus \HH^{*,*}(\Sp V_{2k}(\af^{2n}))\beta_{n-k} \]

where $\beta_{n-k} = \Omega \tau \in \HH^{2(n-k)-1,\{2(n-k)\}}(\Sp V_{2k+2}(\af^{2n}))$. This shows the additive structure of the theorem.

To determine the multiplicative structure, we only need to show that $ \beta_{n-k}^2=0 \in \HH^{4(n-k)-2,\{4(n-k)\}}(\Sp V_{2k+2}(\af^{2n}))$. By considering the degrees as before, we see that the degree of $ \beta_{n-k}^2$ is too small to be the product of two or more generators. Therefore, $ \beta_{n-k}^2=\sum c_j \beta_j $, where $ c_j\in \HH^{4(n-k)-2j-1, \{4(n-k)-2j\}} $. Applying Remark \ref{HHvainish}, if $  4(n-k)-2j > 0 $, then $ 4(n-k)-2j-1>0 $, and we have $c_j=0$. Similarly, if $4(n-k)-2j \leq 0 $, then $ 4(n-k)-2j-1<0 $, and again we have $c_j=0$.
\end{proof}

\section{Appendix}
\label{appendix}

A crucial lemma for matrix computations is that the elementary matrices induce $\af^1$-homotopies that allow to construct retractions. Our aim in this appendix is to give a general unified framework for the constructions we need in the main body of this article. For any $n\in\mathbb{N}$, we set $M_n:=\af^{n\times n}$ and we consider the matrix multiplication 
\[
M_n\times M_n\to M_n.
\]

We say that a subvariety $C\subset M_n$ is \emph{affine contractible} if it is a submonoid for the matrix multiplication and there exists an integer $m\in\mathbb{N}$ and an embedding $f_C: \af^{m} \to M_n$ such that $f_C(0)=\mathrm{Id}$, and $C= f_C(\af^{m})$. In other words, $C$ is an image of affine embedding to $M_n$ containing the identity. In particular, we are interested in the affine contractible subvariety consisting of computations of elementary matrices.

\begin{lemma}[Matrix Retraction]
	\label{Mretract}
	Let $n\in\mathbb{N}$, let $X$ be a variety and let $A\subset X$ be a subvariety. Suppose that there exists an affine contractible subvariety $C\subset M_n$ acting on $X$, and that there exists a morphism $T: X\to C$ such that $xT(x)\in A$ and $T(a)=\mathrm{Id}$ for all $a\in A$. Then, the morphism $r:X\to A$ defined by $x\mapsto xT(x)$ satisfies $r\circ i=\mathrm{Id}_A$ and $i\circ r$ is homotopic to $\mathrm{Id}_X$. Consequently, $ A\cong X $ in $\mathcal{H}(S)$.
\end{lemma}
\begin{proof}
	We only need to show that the morphism $r': X\to X$ defined by $x\mapsto xT(x)$ is $\af^1$-homotopic to the identity. We just set $T(x)=f_C(\mathbf{a}_x)$ and define a homotopy $F: X\times \af^1 \to X$ by $(x,t)\mapsto xf_C(t\mathbf{a}_x) $. A direct computation shows that $F$ does the trick.
\end{proof}

To fix the ideas, recall that the symplectic group of size $2n$ is defined as the subgroup of invertible matrices satisfying

\[
A^TJ_{2n}A=J_{2n}
\]
where $J_{2}=\begin{bmatrix}
	0&1\\
	-1& 0
\end{bmatrix}$ and $J_{2n}:=\mathrm{diag}(J_2,J_{2n-2})$ is defined by induction.
Now we consider the morphism $p: \Sp_{2n}\to \af^{2n-1}$ defined by $p(A)=(A_{12},\ldots, A_{1,2n})$.
\begin{lemma}
	We have an isomorphism $p^{-1}(0)\cong \Gm \times \Sp_{2n-2}\times \af^{2n-1}$, yielding a weak equivalence $p^{-1}(0)\cong \Gm \times \Sp_{2n-2} $.
\end{lemma}
\begin{proof}
	Let $A\in p^{-1}(0)$ be of the form 
	\[ \begin{bmatrix}
		\begin{bmatrix}
			\lambda & 0\\
			a & b
		\end{bmatrix} &
		\begin{bmatrix}
			0\\
			\mathbf{v}
		\end{bmatrix}\\
		\begin{bmatrix}
			\mathbf{w}_1^T & \mathbf{w}_2^T\\
		\end{bmatrix} & M
	\end{bmatrix} \]
	Using the fact that $A^TJ_{2n}A=J_{2n}$, we obtain that $M\in \Sp_{2n-2}$ together with equations
	\[\lambda b J_2+ \begin{pmatrix}
		\mathbf{w}_1\\
		\mathbf{w}_2
	\end{pmatrix} J_{2n-2}\begin{pmatrix}
		\mathbf{w}_1^T & \mathbf{w}_2^T\\
	\end{pmatrix} = J_2 
	\] 
	and
	\[ \begin{pmatrix}
		\lambda \mathbf{v}\\
		0
	\end{pmatrix}+
	\begin{pmatrix}
		\mathbf{w}_1\\
		\mathbf{w}_2
	\end{pmatrix} J_{2n-2}M=0.
		\]
	By the second equality, we can conclude that 
	\[\mathbf{w}_1= -\lambda \mathbf{v} (J_{2n-2}M)^{-1}, \mathbf{w}_2=0\]
	Therefore 
	\[\begin{pmatrix}
		\mathbf{w}_1\\
		\mathbf{w}_2
	\end{pmatrix} J_{2n-2}\begin{pmatrix}
		\mathbf{w}_1^T & \mathbf{w}_2^T\end{pmatrix} =0
	\]
	and $ b=\lambda^{-1}$. Hence $A$ is uniquely determined by $\lambda,a,v$ and $M$. 

\end{proof}

Since $p^{-1}(0)$ is a complete intersection in $\Sp_{2n}$, its normal bundle admits an obvious trivialization (i.e. $A_{12}=\ldots=A_{1,2n}=0$). We then obtain 
\[
 \Sp_{2n}/p^{-1}(\afnz{2n-1})\cong p^{-1}(0)_+\wedge \Sigma \afnz{2n-1}.
 \]
and a cofiber sequence 
\begin{equation}\label{eqn:symplecticcofiber} 
p^{-1}(\afnz{2n-1})_+\to \Sp_{2n+} \to \Sp_{2n}/p^{-1}(\afnz{2n-1})\cong (\Gm \times \Sp_{2n-2})_+\wedge \Sigma \afnz{2n-1}
\end{equation}

To decide the homotopy type of $ p^{-1}(\afnz{2n-1})$, we make use of lemma \ref{Mretract}. We consider the affine contractible subvariety $C$ of $\Sp_{2n}$ given by $f: \af^{2n-1}\to \Sp_{2n}$ defined by
\[ (a, \mathbf{w})\mapsto \begin{bmatrix}
	\begin{bmatrix}
		1 & 0\\
		a & 1
	\end{bmatrix} &
	\begin{bmatrix}
		0\\
		-\mathbf{w}J_{2n-2}
	\end{bmatrix}\\
	\begin{bmatrix}
		\mathbf{w}^T & 0\\
	\end{bmatrix} & \mathbf{I}
\end{bmatrix}\] 

\begin{lemma}
	There exists an $\af^1$-homotopy retraction 
\[
r:p^{-1}(\afnz{2n-1}) \to \Sp_{2n-2}\times \afnz{2n-1}
\]
of the inclusion $\Sp_{2n-2}\times \afnz{2n-1}\subset p^{-1}(\afnz{2n-1})$, which is consequently a weak equivalence.
\end{lemma}

\begin{proof}
	Let $Q_{2m-1}\to \afnz{m}$ be the affine bundle 
\[
\Spec K[a_1,\ldots,a_m, x_1,\ldots,x_m]/(\mathbf{a}\cdot\mathbf{x}^T-1)\xr{(a_1,\ldots,a_m)} \afnz{m}.
\] 
For any variety $V$ over $\afnz{m}$, we can consider $QV:= Q_{2m-1}\times_{\afnz{m}}V$.  Since the projection $Q_{2m-1}\to \afnz{m}$ is an affine bundle, the same is true for the projection $QV\to V$ and it follows that the latter is a weak equivalence. This applies in particular to $V=p^{-1}(\afnz{2n-1})$ and we denote by $W_{2n}$ the scheme $Qp^{-1}(\afnz{2n-1})$.

Let $\wt{M}=(\mathbf{x},M)\in W_{2n}$ be such that $M\in p^{-1}(\afnz{2n-1})$ is of the form 
	\[M=\begin{pmatrix}
		a& \mathbf{v}\\
		\mathbf{w}^T & M_1
	\end{pmatrix}\]
with $\mathbf{x}\cdot \mathbf{v}^T=1$. We define the transformation matrix $T_1(\wt{M})\in C$ to be
	\[ \begin{bmatrix}
		\begin{bmatrix}
		1 & 0\\
		(a-1)x_1 & 1
	\end{bmatrix} &
	\begin{bmatrix}
		0\\
		-(a-1)\mathbf{x}_2J_{2n-2}
	\end{bmatrix}\\
	\begin{bmatrix}
		(a-1)\mathbf{x}_2^T & 0\\
	\end{bmatrix} & \mathbf{I}
\end{bmatrix} \]
where $(x_1, \mathbf{x}_2) =\mathbf{x} $. We obtain
\[ M':=MT_1(\wt{M})=\begin{pmatrix}
	1& \mathbf{v}'\\
	\mathbf{w}'^T & M'_1
\end{pmatrix}\]
Since $ \mathbf{v}'=(v_1, \mathbf{v}_2 - v_1(a-1)\mathbf{x}_2J_{2n-2} ) $ is still in $ \afnz{2n-1} $, it follows that $M'\in p^{-1}(\afnz{2n-1}) $. For the same reason, $p^{-1}(\afnz{2n-1})$ is stable under multiplication by $C$.

Similarly, we can define 
\[T_2(M')= \begin{bmatrix}
	\begin{bmatrix}
	1 & 0\\
	-w'_1 & 1
\end{bmatrix} &
\begin{bmatrix}
	0\\
	\mathbf{w'}_2J_{2n-2}
\end{bmatrix}\\
\begin{bmatrix}
	-\mathbf{w'}_2^T & 0\\
\end{bmatrix} & \mathbf{I}
\end{bmatrix}\] and after multiplying, we get
\[ M''= T_2(M')M' = \begin{pmatrix}
	1& \mathbf{v}''\\
	0 & M''_1
\end{pmatrix}\]
As in the previous lemma, we find that $M''$ must have the form
\[ \begin{bmatrix}
	\begin{bmatrix}
		1 & v_1''\\
		0 & 1
	\end{bmatrix} &
	\begin{bmatrix}
		\mathbf{v''}\\
		0
	\end{bmatrix}\\
	\begin{bmatrix}
	0 & \mathbf{u} \\
	\end{bmatrix} & M_2
\end{bmatrix} \]
where $\mathbf{u}=\mathbf{v''} (J_{2n-2}M_2)^{-1}$ and $M_2\in \Sp_{2n-2}$. This also yields an embedding $i: \Sp_{2n-2}\times \afnz{2n-1}\hookrightarrow p^{-1}(\afnz{2n-1}) $. And by our construction, using lemma \ref{Mretract}, we obtain a $\af^1$-homotopy retraction $r:W_{2n} \to \Sp_{2n-2}\times \afnz{2n-1} $.
\end{proof}

Therefore the cofiber sequence \eqref{eqn:symplecticcofiber} now becomes 
\[ (\Sp_{2n-2}\times \afnz{2n-1})_+\to \Sp_{2n+} \to (\Gm \times \Sp_{2n-2})_+\wedge \Sigma \afnz{2n-1}\]

Finally, we consider the following diagram
\[\begin{tikzcd}[column sep=tiny]
	{(\Sp_{2n-2}\times \afnz{2n-1})_+} & {(\Sp_{2n-2}\times \af^{2n-1})_+} & {\Sp_{2n-2+}\wedge\Sigma \afnz{2n-1})} \\
	{(\Sp_{2n-2}\times \afnz{2n-1})_+} & {\Sp_{2n+}} & {(\Gm \times \Sp_{2n-2})_+\wedge \Sigma \afnz{2n-1}} \\
	{(\Sp_{2n-2}\times \af^{2n-1})_+} & {\Sp_{2n+}} & {\Gm\wedge\Sp_{2n-2+}\wedge\Sigma \afnz{2n-1})}
	\arrow[Rightarrow, no head, from=1-1, to=2-1]
	\arrow[from=1-1, to=1-2]
	\arrow[from=1-2, to=1-3]
	\arrow[from=1-2, to=2-2]
	\arrow[from=2-1, to=3-1]
	\arrow[from=3-1, to=3-2]
	\arrow[Rightarrow, no head, from=2-2, to=3-2]
	\arrow[from=2-1, to=2-2]
	\arrow[from=2-2, to=2-3]
	\arrow[dashed, from=1-3, to=2-3]
	\arrow[dashed, from=2-3, to=3-3]
	\arrow[from=3-2, to=3-3]
\end{tikzcd}\]
Using the bottom line, the octahedral axiom, and $\Gm\wedge\Sigma \afnz{2n-1}\cong \afnz{2n} $ we obtain a cofiber sequence
\[ \Sp_{2n-2+}\to \Sp_{2n+} \to \Sp_{2n-2+}\wedge \afnz{2n}.\]

\bibliographystyle{unsrt}
\bibliography{references}

\begin{thebibliography}{10}

\bibitem{morel19991}
Fabien Morel and Vladimir Voevodsky.
\newblock $ \mathbb{A}^1$-homotopy theory of schemes.
\newblock {\em Publications Math{\'e}matiques de l'IH{\'E}S}, 90:45--143, 1999.

\bibitem{bachmann2020milnor}
Tom Bachmann, Baptiste Calm{\`e}s, Fr{\'e}d{\'e}ric D{\'e}glise, Jean Fasel, and Paul~Arne {\O}stv{\ae}r.
\newblock {Milnor-Witt motives}.
\newblock {\em arXiv preprint arXiv:2004.06634}, 2020.

\bibitem{biglari2012motives}
Shahram Biglari.
\newblock Motives of reductive groups.
\newblock {\em American Journal of Mathematics}, 134(1):235--257, 2012.

\bibitem{williams2012motivic}
Ben Williams.
\newblock {The motivic cohomology of Stiefel varieties}.
\newblock {\em Journal of K-theory}, 10(1):141--163, 2012.

\bibitem{hatcher2002algebraic}
A.~Hatcher.
\newblock {\em {Algebraic Topology}}.
\newblock Algebraic Topology. Cambridge University Press, 2002.

\bibitem{panin2022quaternionic}
Ivan Panin and Charles Walter.
\newblock {Quaternionic Grassmannians and Borel classes in algebraic geometry}.
\newblock {\em St. Petersburg Mathematical Journal}, 33(1):97--140, 2022.

\bibitem{asok2018homotopy}
Aravind Asok, Fr{\'e}d{\'e}ric D{\'e}glise, and Jan Nagel.
\newblock {The homotopy Leray spectral sequence}.
\newblock {\em Motivic Homotopy Theory and Refined Enumerative Geometry}, 2020.

\bibitem{garkusha2019reconstructing}
Grigory Garkusha.
\newblock Reconstructing rational stable motivic homotopy theory.
\newblock {\em Compositio Mathematica}, 155(7):1424--1443, 2019.

\bibitem{fasel2023tate}
Jean Fasel and Nanjun Yang.
\newblock On tate milnor-witt motives.
\newblock {\em arXiv preprint arXiv:2307.03424}, 2023.

\bibitem{morel2012a1}
Fabien Morel.
\newblock {\em {{$\mathbb{A}^1$}-algebraic topology over a field}}, volume 2052.
\newblock Springer, 2012.

\bibitem{panin2023algebraic}
Ivan Panin and Charles Walter.
\newblock {On the algebraic cobordism spectra MSL and MSp}.
\newblock {\em St. Petersburg Mathematical Journal}, 34(1):109--141, 2023.

\bibitem{suslin1991k}
Andrei Suslin.
\newblock {K-theory and K-cohomology of certain group varieties}.
\newblock {\em Advances in Soviet Mathematics}, 4:53--74, 1991.

\bibitem{Asok12a}
A.~Asok and J.~Fasel.
\newblock A cohomological classification of vector bundles on smooth affine threefolds.
\newblock {\em Duke Math. J.}, 163(14):2561--2601, 2014.

\bibitem{morel2020cellular}
Fabien Morel and Anand Sawant.
\newblock Cellular a1-homology and the motivic version of matsumoto's theorem.
\newblock {\em Advances in Mathematics}, 434:109346, 2023.

\bibitem{bondarko2017dimensional}
Mikhail Bondarko and Fr{\'e}d{\'e}ric D{\'e}glise.
\newblock Dimensional homotopy t-structures in motivic homotopy theory.
\newblock {\em Advances in Mathematics}, 311:91--189, 2017.

\bibitem{deglise2011modules}
Fr{\'e}d{\'e}ric D{\'e}glise.
\newblock Modules homotopiques.
\newblock {\em Documenta Mathematica}, 16:411--455, 2011.

\bibitem{feld2020milnor}
Niels Feld.
\newblock {Milnor-Witt cycle modules}.
\newblock {\em Journal of Pure and Applied Algebra}, 224(7):106298, 2020.

\bibitem{deglise2022perverse}
Fr{\'e}d{\'e}ric D{\'e}glise, Niels Feld, and Fangzhou Jin.
\newblock {Perverse homotopy heart and MW-modules}.
\newblock {\em arXiv preprint arXiv:2210.14832}, 2022.

\end{thebibliography}

\end{document}